\documentclass{amsart}

\usepackage[utf8]{inputenc}
\usepackage{amsthm}
\usepackage{amssymb}
\usepackage{amsmath}
\usepackage{mathtools}
\usepackage{tikz-cd}
\usepackage{leftidx}
\usepackage{hyperref}

\newcommand{\cC}{\mathcal{C}}
\newcommand{\cG}{\mathcal{G}}

\newcommand{\End}{\mathrm{End}}
\newcommand{\fg}{\mathfrak{g}}
\newcommand{\fh}{\mathfrak{h}}
\newcommand{\fk}{\mathfrak{k}}

\newcommand{\ff}{\mathfrak{f}}
\newcommand{\ndZ}{\mathbb{Z}}
\newcommand{\ot}{\otimes}
\newcommand{\id}{\operatorname{id}}
\newcommand{\Kd}{\delta^{\mathrm{K}}}
\newcommand{\lma}[2]{\leftidx{\vphantom{+}_{#1}}{\alpha}{\vphantom{+}_{#2}}}
\newcommand{\lmC}[1]{\prescript{}{#1}\cC}

\newcommand{\lcmC}[1]{\prescript{#1}{}\cC}

\newcommand{\lcrmC}[1]{\prescript{#1}{#1}\cC}

\newcommand{\crmh}{\widehat{\eta}} 
\newcommand{\Lone}[1][\cG]{L(1,#1)}
\newcommand{\Loone}[1][\cG]{L_0(1,#1)}

\makeatletter
\numberwithin{equation}{section}
\numberwithin{figure}{section}
\numberwithin{table}{section}
\newtheorem{thm}{Theorem}[section]
\newtheorem*{thm*}{Theorem}
\newtheorem{lem}[thm]{Lemma}

\newtheorem{pro}[thm]{Proposition}

\theoremstyle{definition} 
\newtheorem{defn}[thm]{Definition}

\newtheorem{rem}[thm]{Remark}

\title[Bosonization of curved Lie bialgebras]{Bosonization of curved Lie bialgebras}
\author{Istvan Heckenberger}
\address[I. Heckenberger]{Philipps-Universit\"at Marburg,
        FB Mathematik und Informatik,
        Hans-Meer\-wein-Stra\ss e,
35032 Marburg, Germany.}
\email{heckenberger@mathematik.uni-marburg.de}

\author{Leandro Vendramin}
\address[L. Vendramin]{Department of Mathematics and Data
Science, Vrije Universiteit Brussel, Pleinlaan 2, 1050 Brussel}
\email{Leandro.Vendramin@vub.be}

\date{\today}

\begin{document}


\maketitle

\begin{abstract}
    We use Cartier's preadditive symmetric monoidal categories to 
    study Lie bialgebras. We prove that bosonization can be done consistently in this framework. In the last part of the paper we present explicit examples and indicate a deep relationship between certain curved Lie bialgebras and Nichols algebras over abelian groups.
\end{abstract}

\section{Introduction}

Since their first appearance, Lie bialgebras and their generalizations fascinate a large community of mathematicians and physicists. Motivated by Sklyanin's classical $r$-matrix, Drinfeld  \cite{MR688240,MR1047964} and Semenov-Tian-Shansky  \cite{MR725413} used Lie bialgebras to study Poisson Lie groups.
Publications, problems and questions of Drinfeld had a great impact on further development of the field. Between 1996 and 2008, Etingof and Kazhdan proved in a series of papers that Lie bialgebras can be quantized, with particular emphasis on Kac-Moody and vertex algebras \cite{MR1403351,MR1669953,MR1771217,MR1771218,MR2452604}. One of the main ideas of their approach was to quantize the whole category of modules rather than the universal enveloping algebra of the Lie algebra only, and then to conclude the existence of a quantized (quasi-)Hopf algebra using Tannakian 
reconstruction.

The deformation theory of Lie bialgebras remained until today a vital area of research. The class of examples has been extended successively among others by Andruskiewitsch, Enriquez, Geer, Halbout, Hurle, Majid, and Makhlouf to include quasi-Lie bialgebras, $\Gamma$-Lie bialgebras, Lie superbialgebras, color Lie bialgebras, and braided Lie bialgebras. For more details we refer to the papers \cite{MR1227871,MR2173841,MR2407849,MR2630065, MR2629982, MR2264064,MR4008967,MR1744574}. A classification of Lie bialgebras over current algebras was achieved in \cite{MR2734334}.

Parallelly, several attempts have been made to study (among others) Lie algebras in very general categorical contexts,
e.g.~\cite{MR3761992,MR3070680,MR3839605,MR3154807}. A good source towards this is the work \cite{MR3322335} of Buchberger and Fuchs, where basic concepts on Lie algebras in (pre)additive symmetric monoidal categories have been worked out pleasingly in detail. Although there seems to be no clear agreement about the setting in the largest possible generality, there is a clear desire to search for this.
E.g.~Goyvaerts and Vercruysse write in \cite{MR3070680}:
\textit{Motivated by the way that the field of Hopf algebras benefited from the
interaction with the field of monoidal categories (see e.g. \cite{MR1800719}) on one hand, and
the strong relationship between Hopf algebras and Lie algebras on the other hand,
the natural question arose whether it is possible to study Lie algebras within the
framework of monoidal categories, and whether Lie theory could also benefit from
this viewpoint.}
In the same vein, Buchberger and Fuchs write:
\textit{\dots we insist on imposing relevant conditions directly on the underlying
category. \dots a benefit of the abstraction inherent in the categorical
point of view is that it allows one to neatly separate features that apply only to a subclass of
examples from those which are essential for the concepts and results in question and are thereby
generic.}

Our work is built partially on the point of view illustrated in the previous paragraph. As already observed by Majid in \cite[Def.\,2.2,\,Th.\,3.7]{MR1744574}, the concept of a semidirect sum (also called bisum by Majid) of Lie bialgebras requires the introduction of a significantly more general framework and the appending of an additional term in the Lie bialgebra axiom. Motivated by the bosonization theory of braided Hopf algebras in the largest currently known categorical setting, see 
e.g. \cite[Ch.~3]{MR4164719}, we use 
the notion of a Cartier category, which is nothing but a preadditive symmetric monoidal category together with an infinitesimal braiding. This gives us a very general  categorical framework and, quite surprisingly to the authors, allows a consistent discussion of bosonization. It should also be mentioned that most ideas regarding bosonization, except the use of the infinitesimal braiding, are already available in \cite{MR1744574}. Majid himself writes that due to lack of examples (this was around the year 2000) it is not clear to him which framework is most suitable for his presentation. Also, Majid is not explaining how bosonization works for his braided Lie bialgebras.

Our second main motivation for the present work was an observation on some examples of Nichols algebras over some abelian groups. In \cite{MR4298502} new examples of pointed Hopf algebras of finite Gelfand-Kirilliv dimension appeared and were described very explicitly.
During a visit of Hector Pe\~na Pollastri in February 2022 in Marburg, he and the first named author observed in some of these examples an appearance of Lie bialgebras in that context. We do not aim in this paper a detailed analysis of the precise connection. Nevertheless, we present some examples of Lie bialgebras in some Cartier categories with non-trivial infinitesimal braiding --- more precisely, in the category of crossed modules over some abelian coabelian Lie bialgebras, with possibly non-trivial braiding --- and point out the Nichols algebras they are related~to.

\medskip
The paper is organized as follows. Section~\ref{section:additional} recalls 
Cartier categories. In Sections~\ref{section:lie} 
and~\ref{section:lieco} we recall basic definitions of Lie algebras
and Lie coalgebras in symmetric monoidal categories. Lie bialgebras
in Cartier categories appear in Section~\ref{section:liebi}. 
In Section~\ref{section:bosonization} (see Theorems~\ref{thm:kerpicrossed} and
\ref{thm:doublesum}) we discuss bosonization of Lie bialgebras 
in Cartier categories. Concrete examples (the Jordan plane, the super Jordan plane, and the Laistrygonians)
are discussed in Section~\ref{section:examples}. 

\section{Symmetric monoidal categories with additional structure}
\label{section:additional}

In this paper, Lie bialgebras and variations of them will be objects in symmetric monoidal categories with additional structure. For convenience we will assume that the category is strict.
In the literature such an assumption is not unusual, see e.g.~\cite{MR3322335}.
In this section the most important properties of such categories are collected, which will be used freely in the sequel. Typically we write $\tau$ for the symmetry of such a category 
and use occasionally the leg notation:
$$ \tau_{i(i+1)}=\id^{\ot i-1}\ot \tau \otimes \id ^{\ot j}\in \End (V^{\ot i+j+1}) $$
for all objects $V$ and all $i\ge 1$, $j\ge 0$.

Recall that a preadditive category is a category where the morphisms between any two objects form an abelian group, and composition of morphisms satisfies the distributive law. A preadditive symmetric monoidal category is a symmetric monoidal category which is preadditive and the tensor functor is additive, that is, tensor product and addition of morphisms satisfy the distributive laws.

For our purpose we will need preadditive symmetric monoidal categories with an additional ingredient.

\begin{defn}
  Let $\cC $ be a preadditive symmetric monoidal category
  with identity object $I$ and symmetry $\tau $, together with a natural transformation 
  $$\eta=(\eta_{X,Y})_{X,Y\in \cC}$$
  from the monoidal functor $\otimes: \cC\times \cC\to \cC$ to itself.
  Assume that $\eta$ satisfies the following conditions:
  \begin{align}
    \label{eq:hs2}
    \eta_{X,Y\ot Z}&=\eta_{X,Y}\ot \id_Z
    +(\tau_{Y,X}\ot \id_Z)(\id_Y\ot \eta_{X,Z})(\tau_{X,Y}\ot \id_Z),\\
    \label{eq:hs4}
    \eta _{Y,X}\tau _{X,Y}&=\tau _{X,Y}\eta_{X,Y}
  \end{align}
  for all $X,Y,Z\in \cC$.
  
  Then we say that $(\cC,\eta)$ is a \textbf{Cartier category}, and following 
  Cartier we 
  call $\eta$
  its \textbf{infinitesimal braiding}.
\end{defn}

The notion goes back to Cartier \cite[\S4]{MR1331627}. In
\cite[Section XX.4]{MR1321145} these categories (satisfying a mild additional assumption) are called \emph{infinitesimal symmetric
categories}.

Note that each preadditive symmetric monoidal category is Cartier with
$$ \eta_{X,Y}=0$$
for all $X,Y\in \cC$. We are going to show how Lie bialgebras and bosonization give rise to Cartier categories with non-zero infinitesimal braiding.

The following remarks follow more or less directly from the definitions; they are probably well-known.

\begin{rem}\
  \begin{enumerate}
      \item The axioms of a Cartier category imply in particular that
  \begin{align} \label{eq:hs1}
    \eta_{I,X}&=0=\eta_{X,I},\\
    \label{eq:hs3}
    \eta_{X\ot Y,Z}&=\id_X\ot \eta_{Y,Z}
    +(\id_X\ot \tau_{Z,Y})(\eta_{X,Z}\ot \id_Y)(\id_X\ot \tau_{Y,Z})
  \end{align}
  for all $X,Y,Z\in \cC$. Indeed, Equation~\eqref{eq:hs2} with $Y=Z=I$ says that
  $$ \eta_{X,I}=\eta_{X,I}+\eta_{X,I} $$
  and hence $\eta_{X,I}=0$. The rest follows by using also Equation~\eqref{eq:hs4}.
  
  \item Let $\cC$ be the (preadditive symmetric monoidal) category Vec with the flip as symmetry. Let $\eta $ be a natural transformation from the monoidal functor of $\cC$ to itself. By using rank one linear maps between (non-zero) vector spaces, it follows quickly that there exists a scalar $\lambda $ such that $\eta_{X,Y}=\lambda \id_{X\otimes Y}$  for all vector spaces $X,Y$. Thus \eqref{eq:hs1} implies that the only infinitesimal braiding of Vec is the zero natural transformation.

    \item Let $\Bbbk$ be a field, let $G$ be an abelian monoid, and let $\cC={}^{\Bbbk G}\mathcal{M}$ be the (preadditive symmetric monoidal) category of $\Bbbk G$-comodules (or, equivalently, $G$-graded vector spaces over $\Bbbk$). Again, the symmetry of the category is the flip. By generalizing the arguments for Vec, one concludes that the natural transformations from the monoidal functor of $\cC$ to itself correspond to maps $\chi:G\times G\to \Bbbk$ such that
$$ \eta_{X,Y}=\chi(g,h)\id_{X\ot Y} $$
for all $g,h\in G$ and homogeneous objects $X$ of $G$-degree $g$ and $Y$ of degree~$h$.

In this setting, $\eta $ is an infinitesimal braiding, that is, a natural transformation satisfying Equations \eqref{eq:hs2} and \eqref{eq:hs4}, if and only if $\chi $ is a symmetric additive bicharacter of $G$ with values in $\Bbbk$, that is,
$$ \chi (h_1,h_2)=\chi(h_2,h_1),\quad \chi(g,h_1+h_2)=\chi(g,h_1)+\chi(g,h_2)$$
for all $g,h_1,h_2\in G$.
  \end{enumerate}
\end{rem}

\section{Lie algebras in symmetric categories}
\label{section:lie}

Let $\cC$ be a preadditive symmetric monoidal category.
A \textbf{Lie algebra in} $\cC$ is a pair $(\fg,\beta )$, where $\fg $ is an object in $\cC$ and $\beta:\fg \ot \fg \to \fg$ is a morphism in $\cC $ such that
\begin{enumerate}
    \item $\beta (\id +\tau)=0$ ($\beta $ is antisymmetric) and
    \item $\beta (\beta \ot \id -(\id \ot \beta )(\id -\tau_{12}))=0$
    (Jacobi identity).
\end{enumerate}
If $\cC$ is Cartier, then a Lie algebra in $\cC$ is just a Lie algebra in the underlying preadditive symmetric monoidal category.
 
The categorical context implies directly some compatibility conditions between the braiding $\tau $ and the bracket $\beta$ of a Lie algebra $\fg$ in $\cC$. We will typically use freely these identities. Two of the most frequent identities are
\[ \tau(\beta \otimes \id)(\id \ot \tau )=(\id \otimes \beta )(\tau \ot \id ): \fg  \ot V\ot \fg \to V\ot \fg \]
and
\[
\tau(\id \otimes \beta )(\tau \ot \id )=(\beta \ot \id )(\id \ot \tau ): \fg  \ot V\ot \fg \to \fg \ot V 
\]
for all objects $V\in \cC$.

Another consequence is the following equivalent form of the Jacobi identity:
\begin{align} \label{eq:Jacobieq}
\beta(\id \ot \beta)(\id +\tau_{23}\tau_{12}+\tau_{12}\tau_{23})=0.
\end{align}

Let $(\ff ,\beta)$ be a Lie algebra in $\cC$.
A  \textbf{left Lie module over $\ff$ in} $\cC$
is a pair $(V,\lma{}{})$, where $V$ is an object in $\cC$ and $\lma{}{}:\ff\ot V\to V$ is a morphism in $\cC$ such that
the diagram
\[
\begin{tikzcd}
	{\ff\ot\ff\ot V} & & {\ff\ot \ff \ot V} & & \ff \ot V \\
	\ff \ot V & & & & V
	\arrow["\id-\tau_{12}", from=1-1, to=1-3]
	\arrow["\id \ot \lma{}{}", from=1-3, to=1-5]
	\arrow["{\beta\ot\id}"', from=1-1, to=2-1]
	\arrow["{\lma{}{}}", from=1-5, to=2-5]
	\arrow["{\lma{}{} }", from=2-1, to=2-5]
\end{tikzcd}
\]
commutes. If we want to be more explicit, we use a notation for $\lma{}{} $ indicating  $V$,  i.e.\
$\lma{}{}=\lma{}{V}$.
We write $\lmC{\ff}$ for the category where the objects are left Lie modules over $\ff$ in $\cC$, and
the morphisms between two objects $(V,\lma{}{V})$
and $(W,\lma{}{W})$ are the morphisms $f:V\to W$ in $\cC$ with
$$ \lma{}{W}(\id \ot f)=f\lma{}{V}. $$
Since $\cC$ is strict, 
the category $\lmC{\ff}$ is preadditive, strict monoidal and symmetric, where the monoidal structure is given by the diagonal action
\begin{equation}
    \label{eq:alpha_tensor}
    \lma{}{V\ot W}=\lma{}V\ot \id_W+(\id_V\ot \lma{}W)(\tau_{\ff,V}\ot \id_W)
\end{equation}
for any $(V,\lma{}V)$, $(W,\lma{}W)$ in $\lmC{\ff}$, and the braiding of $\lmC{\ff}$ is the braiding $\tau $ of $\cC$.

Note that each Lie algebra $(\ff,\beta)$ in $\cC$ is a left Lie module over $\ff$ in $\cC$ with module structure $\lma{}\ff=\beta $. Therefore each tensor power of $\ff$ is a left Lie module over $\ff$ in $\cC$.

Let $(\ff,\beta_{\ff})$ be a Lie algebra in $\cC$ and let
$(\fg,\beta_{\fg})$ be a Lie algebra in $\lmC{\ff}$ with $\ff$-action $\lma{}{\fg}$. Assume that the biproduct $\fg \oplus \ff$ exists in $\cC$. Then there is a unique morphism
$$ \beta: (\fg \oplus \ff)\ot (\fg \oplus \ff)\to \fg \oplus \ff $$
in $\cC $ such that
\begin{equation}
\label{eq:beta}    
\begin{aligned}
  \beta(\iota_{\ff}\ot \iota_{\ff})&=\iota_{\ff}\beta_{\ff} ,&
  \beta(\iota_{\fg}\ot \iota_{\fg})&=\iota_{\fg}\beta_{\fg},\\
\beta(\iota_{\ff}\ot \iota_{\fg})&=\iota_{\fg}\lma{}{\fg},&
\beta(\iota_{\fg}\ot \iota_{\ff})&=-\iota_{\fg}\lma{}{\fg}\tau_{\fg,\ff} ,
\end{aligned}
\end{equation}
where $\iota_{\ff}:\ff\to \fg \oplus \ff$ and
$\iota_{\fg}:\fg\to \fg \oplus \ff$ are the canonical monomorphisms.
Moreover, $(\fg \oplus \ff,\beta)$ is a Lie algebra in $\cC$
and is called the \textbf{semidirect sum of $\fg$ and $\ff $}.

\section{Lie coalgebras in symmetric categories}
\label{section:lieco}

Let $\cC$ be a preadditive symmetric monoidal category. A \textbf{Lie coalgebra in} $\cC$ is a pair $(\fg,\delta )$, where $\fg $ is an object in $\cC$ and $\delta:\fg \to \fg \ot \fg$ is a morphism in $\cC $ such that
\begin{enumerate}
    \item $(\id +\tau)\delta =0$ ($\delta $ is co-antisymmetric) and
    \item $(\delta \ot \id -(\id -\tau_{12})(\id \ot \delta ))\delta=0$
    (co-Jacobi identity).
\end{enumerate}
Similarly to Lie algebras, identities involving the braiding are often applied without explanation. Here are two frequently used identities for the cobracket $\delta $ of a Lie coalgebra $\fg$:
\[ (\id \ot \tau )(\delta \otimes \id)\tau =(\tau \ot \id )(\id \otimes \delta ):  V\ot \fg \to \fg  \ot V\ot \fg\]
and
\[
(\tau \ot \id )(\id \otimes \delta )\tau=(\id \ot \tau )(\delta \ot \id ): \fg \ot V\to
\fg  \ot V\ot \fg 
\]
for all objects $V\in \cC$.

The study of Lie coalgebras has a long tradition, see e.g.~\cite{MR594993}.

Let $(\ff ,\delta)$ be a Lie coalgebra in $\cC$.
A  \textbf{left Lie comodule over $\ff$ in} $\cC$
is a pair $(V,\lambda )$, where $V$ is an object in $\cC$ and $\lambda  :V\to \ff\ot V$ is a morphism in $\cC$ such that
the diagram
\[
\begin{tikzcd}
	{\ff\ot\ff\ot V} & & {\ff\ot \ff \ot V} & & \ff \ot V \\
	\ff \ot V & & & & V
	\arrow["\id-\tau_{12}", to=1-1, from=1-3]
	\arrow["\id \ot \lambda ", to=1-3, from=1-5]
	\arrow["{\delta \ot\id}"', to=1-1, from=2-1]
	\arrow["\lambda ", to=1-5, from=2-5]
	\arrow["\lambda ", to=2-1, from=2-5]
\end{tikzcd}
\]
commutes. We write $\lcmC{\ff}$ for the category where the objects are left Lie comodules over $\ff$ in $\cC$, and the morphisms between two objects $(V,\lambda_V)$ and $(W,\lambda_W)$ are
the morphisms $g:V\to W$ in $\cC$ such that
$$ \lambda_W g=(\id \ot g)\lambda_V. $$
The category $\lcmC{\ff}$ is preadditive strict monoidal and symmetric, where the monoidal structure is given by the diagonal coaction on tensor products,
\begin{align} \label{eq:lambda_tensor}
    \lambda_{V\ot W}=\lambda_V\ot \id_W+(\tau_{V,\ff}\ot \id_W)(\id_V\ot \lambda_W)
\end{align}
for any $(V,\lambda_V)$, $(W,\lambda_W)$ in $\lcmC{\ff}$, and the braiding of $\lcmC{\ff}$ is the braiding of $\cC$.

Note that each Lie coalgebra $(\fg,\delta)$ in $\cC$ is a left Lie comodule over $\fg$ in $\cC$ with comodule structure $\lambda_{\fg}=\delta $. Therefore each tensor power of $\fg$ is a left Lie comodule over $\fg$ in $\cC$.

Let $(\ff,\delta_{\ff})$ be a Lie coalgebra in $\cC$ and let
$(\fg,\delta_{\fg})$ be a Lie coalgebra in $\lcmC{\ff}$ with $\ff$-coaction $\lambda_{\fg}$. Assume that the biproduct $\fg \oplus \ff$ of $\fg$ and $\ff$ exists in $\cC$. Then there is a unique morphism
$$ \delta: \fg \oplus \ff\to (\fg \oplus \ff)\ot (\fg \oplus \ff) $$
in $\cC $ such that
\begin{equation}
\label{eq:delta}
\begin{aligned}
  (\pi_{\ff}\ot \pi_{\ff})\delta &=\delta_{\ff}\pi_{\ff} ,&
  (\pi_{\fg}\ot \pi_{\fg})\delta&=\delta_{\fg}\pi_{\fg},\\
(\pi_{\ff}\ot \pi_{\fg})\delta&=\lambda _{\fg}\pi_{\fg},&
(\pi_{\fg}\ot \pi_{\ff})\delta&=-\tau_{\ff,\fg}\lambda_{\fg} \pi_{\fg},
\end{aligned}
\end{equation}
where $\pi_{\fg}:\fg \oplus \ff\to \fg$ and
$\pi_{\ff}:\fg \oplus \ff\to \ff $ are the canonical epimorphisms.
In particular,
\begin{align}
\label{eq:deltaiotaf}
  \delta \iota_\ff&=((\iota_\ff \pi _\ff +\iota_\fg \pi_\fg)\ot (\iota_\ff \pi _\ff +\iota_\fg \pi_\fg))\delta \iota_\ff =(\iota_\ff \ot \iota_\ff)\delta_\ff,\\
\label{eq:deltaiotag}
  \delta \iota_\fg&=(\iota_\fg \ot \iota_\fg)\delta_\fg
  +(\id -\tau)(\iota_\ff\ot \iota_\fg)\lambda_\fg.
\end{align}
Moreover, $(\fg \oplus \ff,\delta)$ is a Lie coalgebra in $\cC$
and is called the \textbf{semidirect sum of $\fg$ and $\ff $}.

\section{Lie bialgebras in Cartier categories}
\label{section:liebi}

Let $\cC=(\cC,\eta)$ be a Cartier preadditive symmetric monoidal category.
A \textbf{Lie bialgebra in} $\cC$ (or \textbf{curved Lie bialgebra}\footnote{The authors would like to thank Abdenacer Makhlouf for suggesting us the denotation curved Lie bialgebra.}) is a triple $(\fg,\beta,\delta)$, where $(\fg,\beta)$ is a Lie algebra in (the underlying preadditive symmetric monoidal category) $\cC$, $(\fg,\delta)$ is a Lie coalgebra in $\cC$, and
\begin{equation}
\label{eq:bialgebra_compatibility}
    \delta \beta =(\id -\tau )(\beta \ot \id )(\id \ot \tau  )(\delta \ot \id )(\id -\tau ) 
    +(\tau -\id)\eta
\end{equation}
as endomorphisms of $\fg\ot \fg$ in $\cC$. The compatibility condition \eqref{eq:bialgebra_compatibility}
has many other equivalent formulations. One of them
is the following:
\begin{align}
\label{eq:bialgebra_compatibility2}
\begin{aligned}
  \delta \beta =&(\beta \ot \id)(\id \ot \delta)
  +(\id \ot \beta )(\tau \ot \id)(\id \ot \delta)\\
  &+(\id \ot \beta )(\delta \ot \id)
  +(\beta \ot \id)(\id \ot \tau)(\delta \ot \id)
  +(\tau-\id )\eta.
\end{aligned}
\end{align}

In particular, it would be possible to replace $(\beta \ot \id)(\id \ot \tau)(\delta \ot \id)$ in
\eqref{eq:bialgebra_compatibility}
by $(\beta \ot \id)(\id \ot \delta)$ or by $(\id \ot \beta)(\delta\ot \id)$. For the upcoming presentation, in particular, in view of the definition of crossed modules below, we found the form in \eqref{eq:bialgebra_compatibility}  to be most convenient.

\begin{rem}
  Our definition of a Lie bialgebra in a Cartier category is a far reaching but very natural generalization of the notion of a Lie bialgebra. In this paper we even take the perspective that a categorical notion of a Lie bialgebra is only possible after fixing an infinitesimal braiding (which also may be zero) for the category.
\end{rem}

Let $(\ff,\beta,\delta)$ be a Lie bialgebra in $\cC$.
A \textbf{(left) crossed module over $\ff$ in} $\cC$ is a triple $(V,\lma{}{},\lambda)$, where $(V,\lma{}{})\in \lmC{\ff}$,
$(V,\lambda)\in \lcmC{\ff}$, and
\begin{equation} 
\label{eq:lcrm}
\lambda \lma{}{} =(\beta \ot \id)(\id \ot \lambda)
+(\id \ot \lma{}{})(\tau_{\ff,\ff}\ot \id)(\id \ot \lambda)
+(\id \ot \lma{}{})(\delta \ot \id)-\eta
\end{equation}
as endomorphisms of $\ff\ot V$ in $\cC$.
Let $\lcrmC{\ff}$ denote the
category of left crossed modules over $\ff $ in $\cC$,
where morphisms are left Lie module and left Lie comodule morphisms in $\cC$. Recall the diagonal action from Equation~\eqref{eq:alpha_tensor} and the diagonal coaction from Equation~\eqref{eq:lambda_tensor} of $\ff$ on tensor products.
The proof of the following lemma is straightforward and is left to the reader.

\begin{lem} \label{le:lcrmCmonoidal}
The category $\lcrmC{\ff}$ is preadditive symmetric monoidal, where the identity is the identity of $\cC$ with zero action and coaction, the action and the coaction of $\ff$ on tensor products are diagonal, and the braiding is the braiding of $\cC$.
\end{lem}

We will improve Lemma~\ref{le:lcrmCmonoidal} in Proposition~\ref{pr:cmfriendly} below.

\begin{rem}\
  \begin{enumerate}
      \item A Lie bialgebra $(\ff,\beta,\delta)$ in $\cC$ is typically not an object in $\lcrmC{\ff}$ via Lie action $\beta$ and Lie coaction $\delta$.

  However, if we compare \eqref{eq:bialgebra_compatibility2} 
  and \eqref{eq:lcrm}, $(\ff,\beta,\delta)$ is an object in $\lcrmC{\ff}$ if and only if
  $(\beta\otimes\id)(\id\otimes\tau)(\delta\otimes\id)+\tau \eta=0$.
 The latter happens if and only if 
\begin{gather}
\eta=(\id\otimes\beta)(\delta\otimes\id).
\end{gather}
It follows that if $(\ff,\beta,\delta)$ is an object in $\lcrmC{\ff}$, then $\delta\beta=(\id-\tau)\eta $.

  \item Let $(\ff,\beta,\delta)$ be a Lie bialgebra in $\cC$ and let $(V,\alpha_V,\lambda _V)\in \lcrmC{\ff}$. Then $$ (\ff \ot V,\alpha_{\ff \ot V})\in \lmC{\ff},\qquad
  (\ff \ot V,\lambda_{\ff \ot V})\in \lcmC{\ff},$$
  and
  \begin{align*}
    \lambda_V \alpha_V-(\id \ot \alpha_V)\lambda_{\ff \ot V}=(\beta \ot \id)(\id \ot \lambda)-\eta_{\ff,V},\\
    \lambda_V \alpha_V-\alpha_{\ff \ot V}(\id \ot \lambda_V)=(\id \ot \alpha)(\delta \ot \id)-\eta_{\ff,V}.
  \end{align*}
  In particular, typically $\alpha_V:\ff \ot V\to V$ is not a morphism in $\lcmC{\ff}$ and
  $\lambda _V:V\to \ff \ot V$ is not a morphism in $\lmC{\ff}$.
  \end{enumerate}
\end{rem}

\begin{lem} \label{le:zetatensor}
Let $(\ff,\beta,\delta)$ be a Lie bialgebra in the Cartier  category $\cC$.
For each pair $(V,\alpha_V,\lambda_V)$ and $(W,\alpha_W,\lambda_W)$ of objects in $\lcrmC{\ff}$ let
  $$ \zeta_{V,W}=(\lma{}{W}\ot \id_V)(\id_\ff \ot \tau _{V,W})(\lambda_V \ot \id _W):V \ot W \to W\ot V. $$
  \begin{enumerate}
      \item For all morphisms $f:V_1\to V_2$, $g:W_1\to W_2$ in $\lcrmC{\ff}$,
      $$ \zeta_{V_2,W_2}(f\ot g)=(g\ot f)\zeta_{V_1,W_1}. $$
      \item For all $X,Y,Z\in \lcrmC{\ff}$,
      \begin{align*}
        \zeta_{X,Y\ot Z}&=
        (\id _Y\ot \tau_{X,Z}) (\zeta_{X,Y}\ot \id_Z)
        +(\id_Y\ot \zeta_{X,Z}) (\tau_{X,Y}\ot\id_Z),\\
        \zeta_{X\ot Y,Z}&=
        (\zeta_{X,Z}\ot \id_Y)(\id _X\ot \tau_{Y,Z})
        +(\tau_{X,Z}\ot \id_Y)(\id _X\ot \zeta_{Y,Z}).
      \end{align*}
  \end{enumerate}
\end{lem}
  
Note that the morphisms $\zeta_{V,W}$ in the lemma are typically not morphisms in $\lcrmC{\ff}$, not even if $\eta=0$. Nevertheless they are useful to discuss the morphisms in Lemma~\ref{le:zeta}(2) below.

\begin{proof}
  Both claims follow directly from the definitions (including the definitions of the $\ff$-action and of the $\ff$-coaction on a tensor product of two objects) and the naturality of the symmetry $\tau $.
\end{proof}

\begin{lem} \label{le:zeta}
  Let $(\ff,\beta ,\delta )$ be a Lie bialgebra in the Cartier category 
  $\cC$. For each pair $(V,\lma{}{V},\lambda_V)$ and 
  $(W,\lma{}{W},\lambda_W)$ in $\lcrmC{\ff}$ let
\begin{align*}
    \zeta_{V,W}&=(\lma{}{W}\ot \id_V  )
    (\id _\ff\ot \tau _{V,W})(\lambda_V \ot \id _W):
    V\ot W\to W\ot V\\
    \intertext{(as in Lemma~\ref{le:zetatensor}),}
    \hat{\alpha}_{V,W}&=(\lma{}{V}\ot \lma{}{W})
    (\id _\ff\ot \tau _{\ff ,V}\ot \id_W)
    (\delta \ot \id _{V\ot W}):
    \ff \ot V\ot W\to V\ot W,\\
    \hat{\lambda}_{V,W}&=(\beta \ot \id_{V\ot W})
    (\id _\ff\ot \tau _{V,\ff }\ot \id_W)
    (\lambda_V\ot \lambda_W):
    V\ot W\to \ff \ot V\ot W.
\end{align*}
  \begin{enumerate}
      \item
For all $(V,\lma{}{V},\lambda_V),
(W,\lma{}{W},\lambda_W)\in \lcrmC{\ff}$ the following equations hold.
  \begin{align*}
    \tau_{V,W}\hat{\alpha}_{V,W}=
    -\hat{\alpha}_{W,V}(\id_\ff \ot \tau_{V,W}), \qquad 
    \hat{\lambda}_{W,V}\tau_{V,W}&=
    -(\id_\ff\ot\tau_{V,W})\hat{\lambda}_{V,W},\\
    \zeta_{V,W} \lma{}{V\ot W}
    +\tau_{V,W}\hat{\alpha}_{V,W}
    +(\alpha\ot\id)(\id\ot\tau)(\eta_{\ff,V}\ot\id_W)
    &=
    \lma{}{W\ot V}(\id_\ff \ot \zeta_{V,W})
    ,\\
    \lambda_{W\ot V}\zeta_{V,W}+
    \hat{\lambda}_{W,V}\tau _{V,W}
    +(\eta\ot\id)(\id\ot\tau)(\lambda_V\ot\id_W)&=
        (\id _\ff\ot \zeta_{V,W})\lambda_{V\ot W}
    .
  \end{align*}
  \item
    The morphism
    $$ \crmh_{V,W}=\zeta_{W,V}\tau_{V,W}
  +\tau_{W,V}\zeta_{V,W}+\eta_{V,W}
  $$
  in $\cC$ is an endomorphism of $V\ot W$ in $\lcrmC{\ff}$.
  \end{enumerate}
\end{lem}

\begin{proof}
(1) The equations
$$ \tau_{V,W}\hat{\alpha}_{V,W}
=-\hat{\alpha}_{W,V}(\id_\ff \ot \tau_{V,W}),
\qquad
  \hat{\lambda}_{W,V}\tau_{V,W}
  =-(\id_\ff \ot \tau _{V,W})\hat{\lambda}_{V,W}
$$
follow from $\tau \delta=-\delta$ and
$\beta \tau=-\beta$, respectively, and from the naturality of the symmetry $\tau $.

Now we are going to prove the third equation.
Recall that
$$ \lma{}{V\ot W}=\lma{}{V}\ot \id_W
+(\id _V\ot \lma{}{W})(\tau_{\ff,V}\ot \id_W).$$
Therefore
\begin{align*}
  \zeta_{V,W}\alpha_{V\ot W}&=
  (\lma{}{W}\ot \id_V)(\id _\ff\ot \tau _{V,W})(\lambda_V \ot \id _W)\\
  &\qquad \cdot \big(\lma{}{V}\ot \id_W
+(\id _V\ot \lma{}{W})(\tau_{\ff,V}\ot \id_W)\big)\\
&=(\lma{}{W}\ot \id_V)(\id _\ff \ot \tau _{V,W})
\big(
 -\eta_{\ff,V}\ot \id_W+
 (\beta \ot \id_{V\ot W})(\id _\ff \ot \lambda _V\ot \id_W)\\
 &\qquad +(\id _\ff \ot \lma{}{V}\ot \id_W)(\tau_{\ff,\ff }\ot \id_{V\ot W})(\id _\ff \ot \lambda_V\ot \id_W)\\
 &\qquad +(\id _\ff \ot \lma{}{V}\ot \id_W)(\delta \ot \id _{V\ot W})
 +(\lambda_V\ot \alpha_W)(\tau_{\ff,V} \ot \id_W)\big)\\
 \intertext{by the crossed module axiom \eqref{eq:lcrm}. In the second term we rewrite $\alpha_W(\beta \ot \id_W)$ using the Lie module axiom for $(W,\alpha_W)$ and obtain that}
 \zeta_{V,W}\alpha_{V\ot W}&=
 (\alpha_W\ot \id_V)(\id_\ff \ot \tau_{V,W})(-\eta_{\ff,V}\ot \id_W)\\
 &\quad+(\lma{}{W}\ot \id_V)(\id_\ff \ot \lma{}{W}\ot \id_V)((\id -\tau_{\ff,\ff})\ot \tau_{V,W})
 (\id_\ff \ot \lambda_V\ot \id_W)\\
 &\quad +(\id_W\ot \lma{}{V})
 (\tau_{\ff ,W}\ot \id_V)(\id _\ff\ot \zeta_{V,W})\\
 &\quad +(\lma{}{W}\ot \lma{}{V})(\id _\ff \ot \tau_{\ff , W}\ot \id_V)(\delta \ot \id_{W\ot V})(\id_\ff \ot \tau_{V,W})\\
&\quad +(\alpha_W\ot \id_V)(\id_\ff \ot \lma{}{W}\ot \id_V)(\tau_{\ff,\ff}\ot \tau_{V,W})(\id_\ff \ot \lambda_V  \ot \id_W).
\end{align*}    
Now the last term cancels with part of the second term, and the fourth term is $\hat{\alpha}_{W,V}(\id _\ff \ot \tau_{V,W})=-\tau_{V,W}\hat{\alpha}_{V,W}$.
It follows that
\begin{align*}
  \zeta_{V,W}\alpha_{V\ot W}&=
  (\alpha_W\ot \id_V)(\id_\ff \ot \tau_{V,W})(-\eta_{\ff,V}\ot \id_W)\\
  & \quad +(\alpha_W\ot \id_V)(\id_\ff \ot \alpha_W\ot \id_V)(\id \ot \tau_{V,W})(\id_\ff \ot \lambda_V\ot \id_W)\\
&\quad +(\id_W\ot \lma{}{V})
 (\tau_{\ff ,W}\ot \id_V)(\id _\ff\ot \zeta_{V,W})
 -\tau_{V,W}\hat{\alpha}_{V,W}.
 \end{align*}
 In the last expression, the sum of the second and the third term is $\alpha_{W\ot V}(\id_\ff \ot \zeta_{V,W})$.
 This implies the third equation in part (1) of the lemma.

The fourth equation of part (1) of the lemma can be proven similarly.

(2) We prove that $\crmh_{V,W}$ is an endomorphism of the $\ff$-module $V\ot W$. By definition of $\crmh_{V,W}$ and by (1),
\begin{align*}
    \crmh_{V,W}\lma{}{V\ot W}
    &=(\zeta_{W,V}\tau_{V,W}
    +\tau_{W,V}\zeta_{V,W}+\eta_{V,W}) \lma{}{V\ot W}\\
    &=\zeta_{W,V}\lma{}{W\ot V}(\id _\ff \ot \tau_{V,W})
    +\tau_{W,V}(\lma{}{W\ot V}(\id _\ff \ot \zeta_{V,W})
    -\tau_{V,W}\hat{\alpha}_{V,W})\\
    &\qquad +\tau (\alpha \ot \id)(\id \ot \tau)(-\eta_{\ff,V}\ot \id_W)+\eta_{V,W}\alpha_{V\ot W}\\
    &=\lma{}{V\ot W}(\id _\ff \ot \zeta_{W,V}\tau_{V,W})
    +\hat{\alpha}_{V,W}\\
    &\qquad +(\alpha \ot \id)(\id \ot \tau)(-\eta_{\ff,W}\ot \id_V)(\id_\ff \ot \tau_{V,W})\\
    &\qquad +\lma{}{V\ot W}(\id_\ff \ot 
    \tau_{W,V}\zeta_{V,W})
    -\hat{\alpha}_{V,W}\\
    &\qquad +\tau(\alpha \ot \id)(\id \ot \tau)(-\eta_{\ff,V}\ot \id_W)+\eta_{V,W}\alpha_{V\ot W}\\
    &=\lma{}{V\ot W}(\id_\ff \ot \crmh_{V,W}
    -\id_\ff \ot \eta_{V,W})\\
    &\qquad -(\alpha \ot \id)(\id \ot \tau)(\eta_{\ff,W}\ot \id_V)(\id_\ff \ot \tau_{V,W})\\
    &\qquad -\tau (\alpha \ot \id)(\id \ot \tau)(\eta_{\ff,V}\ot \id_W)+\eta_{V,W}\alpha_{V\ot W}.
\end{align*}
The naturality of $\eta$ and the compatibility conditions of $\eta$ and the monoidal structure imply that
\begin{align*}
    \eta_{V,W}\alpha_{V\ot W}
    &=\eta_{V,W}(\alpha_V\ot \id_W+(\id_V\ot \alpha_W)(\tau_{\ff,V}\ot \id_W))\\
    &=(\alpha_V\ot \id_W)\eta_{\ff \ot V,W}
    +(\id_V\ot \alpha_W)\eta_{V,\ff \ot W}(\tau_{\ff,V}\ot\id_W)\\
    &=(\alpha_V\ot \id_W)\big(\id _\ff \ot \eta_{V,W}
    +(\id_\ff \ot \tau_{W,V})(\eta_{\ff,W}\ot \id_V)(\id_\ff \ot \tau_{V,W})\big)\\
    &\qquad +(\id_V \ot \alpha_W)
    \big(\eta_{V,\ff}\tau_{\ff,V}\ot \id_W
    +(\tau_{\ff,V}\ot \id_W)
    (\id_\ff \ot \eta_{V,W})\big).
\end{align*}
After applying the distributive rule, we get four terms.
The first and the fourth terms yield $\alpha_{V\ot W}(\id_\ff \ot \eta_{V,W})$, and the third of the four terms is
\begin{align*}
  (\id_V \ot \alpha_W)(\eta_{V,\ff}\tau_{\ff,V}\ot \id_W)
  &=(\id_V \ot \alpha_W)
  (\tau_{\ff,V}\eta_{\ff,V}\ot \id_W)\\
 &=\tau _{W,V}(\alpha_W\ot \id_V)(\id _\ff \ot \tau_{V,W})(\eta_{\ff,V}\ot \id_W).
\end{align*}
Now it is easy to confirm that $\crmh_{V,W}\alpha_{V\ot W}=\alpha_{V\ot W}(\id_\ff \ot \crmh_{V,W})$.

Similarly, $\crmh_{V,W}$ is an endomorphism of the $\ff$-comodule $V\ot W$, which proves the claim in (2).
\end{proof}

\begin{pro} \label{pr:cmfriendly}
  Let $(\ff,\beta,\delta)$ be a Lie bialgebra in the Cartier category $\cC$, and for all $V,W\in \lcrmC{\ff}$ let $\crmh_{V,W}$ be the endomorphism of $V\ot W\in \lcrmC{\ff}$ from Lemma~\ref{le:zeta}(2). Then
  $\lcrmC{\ff}$ is a Cartier category with infinitesimal braiding $\crmh$.
\end{pro}

\begin{proof}
  As noted in Lemma~\ref{le:lcrmCmonoidal}, $\lcrmC{\ff}$ is a preadditive symmetric monoidal category.
  The naturality of $\crmh$ follows directly from Lemma~\ref{le:zetatensor}(1) and the naturality of $\tau $ and $\eta$.
  Thus it remains to verify Equations~\eqref{eq:hs2} and \eqref{eq:hs4} for $\crmh$.
  
  Equation~$\tau \crmh =\crmh \tau$ follows directly from the definition of $\crmh$, since $\tau^2=\id $ and $\tau \eta =\eta \tau$.
  
  We now prove Equation~\eqref{eq:hs2} for $\crmh $.
  Let $X,Y,Z\in \lcrmC{\ff}$. Then
  \begin{align*}
    \crmh_{X,Y\ot Z}&=
    \zeta _{Y\ot Z,X}\tau _{X,Y\ot Z}
    +\tau _{Y\ot Z,X}\zeta_{X,Y\ot Z}
    +\eta_{X,Y\ot Z}.\\
    \intertext{Then Lemma~\ref{le:zetatensor}(2), the braiding axioms for $\tau $, and Equation~\eqref{eq:hs2} for $\eta$ imply that}
    \crmh_{X,Y\ot Z}&=\big( (\zeta \ot \id)(\id \ot \tau )+(\tau \ot \id)(\id \ot \zeta)\big)(\id \ot \tau )(\tau \ot \id)\\
    &\qquad +(\tau \ot \id)(\id \ot \tau)\big(
    (\id \ot \tau)(\zeta \ot \id)
    +(\id \ot \zeta)(\tau \ot \id)\big)\\
    &\qquad +\eta_{X,Y}\ot \id_Z+
    (\tau _{Y,X}\ot \id_Z)(\id _Y\ot \eta_{X,Z})(\tau _{X,Y}\ot \id_Z)\\
    &=\zeta \tau \ot \id +(\tau \ot \id)(\id \ot \zeta \tau )(\tau \ot \id)\\
    &\qquad +\tau \zeta \ot \id 
    +(\tau \ot \id)(\id\ot \tau \zeta )(\tau \ot \id)\\
    &\qquad +\eta_{X,Y}\ot \id_Z+
    (\tau _{Y,X}\ot \id_Z)(\id _Y\ot \eta_{X,Z})(\tau _{X,Y}\ot \id_Z)\\
    &=\crmh_{X,Y}\ot \id_Z+(\tau _{Y,X}\ot \id_Z)(\id _Y\ot \crmh_{X,Z})
    (\tau _{X,Y}\ot \id _Z)
 \end{align*}
 and hence Equation~\eqref{eq:hs2} is fulfilled for $\crmh$.
 Thus the proof of the proposition is completed.
\end{proof}

\begin{rem}
  Note that $\crmh$ in Proposition~\ref{pr:cmfriendly} is in general non-zero, even if $\eta=0$.
\end{rem}

\section{Bosonization of Lie bialgebras in
Cartier categories}
\label{section:bosonization}

In this section let $\cC=(\cC,\eta)$ be a
Cartier category.
We are going to use
semidirect sum Lie algebras and semidirect sum Lie coalgebras, see \eqref{eq:beta} and \eqref{eq:delta} for the corresponding bracket and cobracket, to define bisum Lie bialgebras in $\cC$.
To do so, we will use the infinitesimal braiding of crossed modules from Proposition~\ref{pr:cmfriendly}.
Recall from Proposition~\ref{pr:cmfriendly}, that for each Lie bialgebra $\ff$ in $\cC$, $(\lcrmC{\ff},\crmh)$ is a Cartier category.

\begin{thm} \label{thm:kerpicrossed}
  Let $\pi:\fg\to \ff$, $\gamma:\ff\to \fg$ be Lie bialgebra morphisms between two Lie bialgebras $(\ff,\beta_\ff,\delta_\ff)$ and $(\fg,\beta_\fg,\delta_\fg)$ in $\cC$ such that $\pi \gamma=\id_\ff$. Assume that $\pi$ has a kernel
  $\kappa :\fk\to \fg$ in $\cC$.
  Let $\vartheta :\fg \to \fk $ be the morphism with $\kappa \vartheta =\id_\fg -\gamma \pi $.
  \begin{itemize}
      \item[(a)] The object $\fg\in \cC$ together with the morphisms $\vartheta:\fg \to \fk$, $\pi:\fg \to \ff$, and $\kappa:\fk\to \fg$, $\gamma:\ff\to \fg$ is a biproduct of $\fk $ and $\ff $ in $\cC$.
      \[
\begin{tikzcd}
  \fk & & \fk \\
  & \fg & \\
  \ff & & \ff
  \arrow["\kappa "', from=1-1, to=2-2]
  \arrow["\gamma", from=3-1, to=2-2]
  \arrow["\vartheta"', from=2-2, to=1-3]
  \arrow["\pi", from=2-2, to=3-3]
  \arrow["\id", from=1-1, to=1-3]
  \arrow["\id", from=3-1, to=3-3]
\end{tikzcd}
\]
      \item[(b)]
  The object $\fk \in \cC$ is a left crossed module over $\ff$ in $\cC$, $\fk \in \lcrmC{\ff}$, via $\ff$-action $\alpha_{\fk}$ and $\ff$-coaction $\lambda_{\fk}$, where
  $$ \alpha_\fk=\vartheta \beta_\fg(\gamma \ot \kappa ),\quad
  \lambda_\fk=(\pi \ot \vartheta )\delta_\fg \kappa .
  $$
  Moreover $\fk$ is a Lie bialgebra in $(\lcrmC{\ff},\crmh)$ with bracket $\beta_\fk$ and cobracket $\delta_\fk$, where
  $$ \beta_\fk = \vartheta \beta_\fg (\kappa \ot \kappa),\quad
  \delta_\fk =(\vartheta \ot \vartheta)\delta_\fg \kappa . $$ 

  \item[(c)]
  The Lie algebra $(\fg,\beta)$ is the semidirect sum of $(\fk,\beta_\fk)$ and $(\ff,\beta_\ff)$.
  The Lie coalgebra $(\fg,\delta)$ is the semidirect sum of $(\fk,\delta_\fk)$ and $(\ff,\delta_\ff)$.
    
  \end{itemize}
\end{thm}

\begin{rem}
 In Theorem~\ref{thm:kerpicrossed},
 a morphism is a kernel of $\pi $ if and only if it is a kernel of $\gamma \pi $, since $\pi \gamma =\id_\ff$ by assumption. Moreover, $\gamma\pi $ is an idempotent morphism. Therefore, if $\cC$ is Karoubian (pseudo-abelian), then the assumption in Theorem~\ref{thm:kerpicrossed} on the existence of a kernel of $\pi $ is always fulfilled. We thank the referee for pointing out this interesting detail.
\end{rem}

\begin{proof}
  (a) Existence and uniqueness of $\vartheta $ follow from the equation
  $$ \pi(\id _\fg -\gamma \pi)=0 $$
  and the choice of $\kappa $. Moreover, $\kappa \vartheta \kappa =\kappa $ since $\pi \kappa=0$. Hence
  $\vartheta \kappa =\id_\fk$
  since $\kappa $ is a monomorphism. Similarly,
  $\vartheta \gamma=0$
  since
  $\kappa \vartheta \gamma=(\id -\gamma \pi)\gamma=0$. The equations $\pi\gamma=\id_{\fh}$ and $\pi \kappa=0$ are clear.

  (b) Composing the Jacobi identity of $(\fg ,\beta _\fg)$
  from the left with $\vartheta $ and from the
  right with $\gamma \ot \gamma \ot \kappa $ implies that the pair $(\fk,\alpha_\fk)$ is a left Lie module over $\ff$ in $\cC$:
  \begin{align*}
      \alpha_\fk(\beta_\ff \ot \id_\fk)&=
      \vartheta \beta_\fg (\gamma \ot \kappa )(\beta _\ff \ot \id_\fk)\\
      &=\vartheta \beta_\fg (\beta _\fg \ot \id_\fg)(\gamma \ot \gamma \ot \kappa )\\
      &=\vartheta \beta_\fg (\id_\fg \ot \beta _\fg)(\id -\tau_{12})(\gamma \ot \gamma \ot \kappa )\\
      &=\vartheta \beta_\fg (\id_\fg \ot (\id _\fg -\gamma \pi)\beta _\fg)(\gamma \ot \gamma \ot \kappa )(\id -\tau_{12})\\
      &=\vartheta \beta_\fg (\gamma \ot \kappa \vartheta \beta _\fg)(\id_\ff \ot \gamma \ot \kappa )(\id -\tau_{12})\\
      &=\alpha_\fk (\id_\ff \ot \alpha_\fk )(\id -\tau_{12}).
  \end{align*}
   Similarly, $(\fk ,\lambda_\fk)$ is a left Lie comodule over $\ff $ in $\cC$.
   The compatibility condition~\eqref{eq:lcrm} between $\alpha_\fk$ and $\lambda_\fk$ follows from the Lie bialgebra axiom \eqref{eq:bialgebra_compatibility2} for $\fg$ composed with $\pi \ot \vartheta$ from the left and with $\gamma \ot \kappa $ from the right. In this calculation the naturality of $\eta $ has to be used. We conclude that $(\fk,\alpha_\fk,\lambda_\fk)\in \lcrmC{\ff}$.

   The Jacobi identity for $\beta _\fg$ in the form
   $$ \beta_\fg(\id _\fg \ot \beta_\fg)=
   \beta_\fg(\beta_\fg \ot \id_\fg)(\id -\tau_{23}) $$
   multiplied from the left with $\vartheta $ and from the right with $\gamma \ot \kappa \ot \kappa$, using the antisymmetry of $\beta_\fg$ and equations between $\gamma,\pi,\kappa$, and $\vartheta$ implies that
   $\beta_\fk$ is a morphism in $\lmC{\ff}$.
   By multiplying the Lie bialgebra axiom \eqref{eq:bialgebra_compatibility2} for $\fg$ from the left with $\pi \ot \vartheta$ and from the right with $\kappa \ot \kappa $ we conclude that $\beta_\fk $ is a morphism in $\lcmC{\ff}$.
   Similarly, $\delta_\fk $ is a morphism in $\lcrmC{\ff}$. 

   The antisymmetry of $\beta_\fk$ follows from the antisymmetry of $\beta_\fg$ and the naturality of the symmetry $\tau$:
   $$ \beta_\fk \tau
   =\vartheta \beta_\fg (\kappa \ot \kappa)\tau
   =\vartheta \beta_\fg \tau (\kappa \ot \kappa)
   =-\vartheta \beta_\fg (\kappa \ot \kappa)
   =-\beta_\fk.
   $$
   The Jacobi identity for $\beta_\fk$ can be concluded from the Jacobi identity of $\beta_\fg$, the defining equation of $\vartheta$, and from the equation
   $\pi \beta_\fg (\kappa \ot \kappa)=0$.
   Similarly, $\delta_\fk $ is co-antisymmetric and satisfies the co-Jacobi identity.  Finally, the defining equation of $\vartheta$ and the Lie bialgebra axiom
   \eqref{eq:bialgebra_compatibility} for $\fg$ imply that
   \begin{align*}
     \delta_\fk \beta_\fk &=
     (\vartheta \ot \vartheta)\delta_\fg \kappa \vartheta \beta_\fg (\kappa \ot \kappa)\\
     &=
     (\vartheta \ot \vartheta)\delta_\fg  \beta_\fg (\kappa \ot \kappa)\\     
     &=
     (\id -\tau )(\vartheta \ot \vartheta)(\beta _\fg \ot \id)(\id \ot \tau)(\delta_\fg \ot \id)(\kappa \ot \kappa)(\id -\tau )\\
     &\qquad
     +(\tau-\id )(\vartheta \ot \vartheta)\eta (\kappa \ot \kappa).\\
     \intertext{Now the defining equation of $\vartheta$ is plugged in and the naturality of $\eta$ is used to conclude that}
     \delta_\fk \beta_\fk &=
     (\id -\tau )(\vartheta \ot \vartheta)(\beta _\fg \ot \id)\big((\kappa \vartheta+\gamma \pi )\ot \id_\fg^{\ot 2}\big)(\id \ot \tau)(\delta_\fg \ot \id)(\kappa \ot \kappa)(\id -\tau )\\
     &\qquad
     +(\tau-\id )\eta_{\fk,\fk}\\
     &=
     (\id -\tau )\big((\beta _\fk\ot \id)(\id \ot \tau)(\delta_\fk \ot \id )+(\alpha_\fk\ot \id)(\id \ot \tau)
     (\lambda_\fk\ot \id)\big)(\id -\tau )\\
     &\qquad
     +(\tau-\id )\eta_{\fk,\fk}.\\
     \intertext{Now the definition of $\crmh$ implies that}
     \delta_\fk \beta_\fk &=
     (\id -\tau )(\beta _\fk\ot \id)
     (\id \ot \tau)(\delta_\fk \ot \id )(\id -\tau )+(\tau-\id )\crmh_{\fk,\fk}.
   \end{align*}
   Thus $\fk $ satisfies Equation~\eqref{eq:bialgebra_compatibility}.

   (c) This follows from the definitions of the brackets, cobrackets, actions and coactions and the biproduct structure morphisms.
   \end{proof}

Now we prove a converse of Theorem~\ref{thm:kerpicrossed}.

\begin{thm}
\label{thm:doublesum}
    Let $\ff$ be a Lie bialgebra in $\cC$ and 
    $\fk$ be a Lie bialgebra in $(\lcrmC{\ff},\crmh)$. Assume that the biproduct $\fk \oplus \ff$ exists in $\cC$. Then 
    $(\fk\oplus\ff,\beta,\delta)$ is a Lie bialgebra in $\cC$, where $(\fk\oplus\ff,\beta)$ is the semidirect sum Lie algebra in $\cC$ and $(\fk\oplus\ff,\delta)$ is the semidirect sum Lie coalgebra in $\cC$.
    This Lie bialgebra is called the \textbf{bisum Lie bialgebra of $\fk $ and $\ff$ in $\cC$}.
\end{thm}

Note that the bisum Lie bialgebra in Theorem~\ref{thm:doublesum} satisfies the assumptions in Theorem~\ref{thm:kerpicrossed}
with $\fg=\fk\oplus \ff$ and with $\pi:\fg\to \ff$, $\gamma: \ff\to \fg$ the corresponding biproduct morphisms.

\begin{proof} It only remains to prove that $\beta $ and $\delta $ satisfy the Lie bialgebra axiom \eqref{eq:bialgebra_compatibility}. Since $\beta \tau=-\beta $ and $(\id-\tau)\tau=-(\id-\tau)$,
it suffices to prove that \eqref{eq:bialgebra_compatibility} is satisfied on $(\fk\ot \fk)\oplus (\ff \ot \fk) \oplus (\ff \ot \ff)$.

Since
$\delta \iota_\ff=(\iota_\ff\ot \iota_\ff)\delta_\ff$ by
\eqref{eq:deltaiotaf} and
$\beta(\iota_\ff\ot \iota _\ff)=\iota_\ff \beta_\ff$
by \eqref{eq:beta}, and since $(\ff,\beta_\ff,\delta_\ff)$ is a Lie bialgebra in $\cC$,
\eqref{eq:bialgebra_compatibility} is clearly satisfied on $\ff\ot \ff$.
Moreover, on both sides of \eqref{eq:bialgebra_compatibility}
the expression is invariant under composition with $-\tau $ from the left or from the right. Therefore it remains to prove that \eqref{eq:bialgebra_compatibility} or \eqref{eq:bialgebra_compatibility2} is satisfied after composing from the left with $\pi_\ff \ot \pi_\fk $ or with $\pi_\fk\ot \pi_\fk $, and from the right with $\iota_\ff \ot \iota_\fk $ or with $\iota_\fk \ot \iota_\fk $.

Let $\lma{}\fk$ and $\lambda_\fk$ denote the left Lie action on $\fk$ by $\ff$ and the left Lie coaction on $\fk$ by $\ff$, respectively.
Then we obtain from
\eqref{eq:beta} and \eqref{eq:delta} that
\begin{align*}
 (\pi_\fk \ot \pi_\fk)\delta \beta (\iota_\fk \ot \iota_\fk)
&=\delta _\fk\,\pi_\fk \,\iota_\fk \,\beta_\fk 
=\delta_\fk \,\beta_\fk.
\end{align*}
On the other hand, using further \eqref{eq:beta} and \eqref{eq:delta}
and the naturality of $\eta$,
we obtain that
\begin{align*}
(\pi_\fk \ot \pi_\fk)&(\id -\tau )(\beta \ot \id)
(\id \ot \tau)(\delta\ot \id )(\id -\tau)
(\iota_\fk \ot \iota_\fk)\\
&\qquad
+(\pi_\fk \ot \pi_\fk)(\tau-\id)\eta_{\fg,\fg}(\iota_\fk \ot \iota_\fk)
\\
&=(\id -\tau)(\pi_\fk \beta \ot \pi_\fk)
( (\iota_\fk \pi_\fk +\iota_\ff \pi_\ff ) \ot \tau)(\delta\iota_\fk \ot \iota_\fk)(\id -\tau)
+(\tau-\id)\eta_{\fk,\fk}
\\
&=(\id -\tau)(\beta_\fk \ot \id_\fk)
( \id \ot \tau)(\delta_\fk \ot \id_\fk)(\id -\tau)\\
&\qquad +
(\id -\tau)(\lma{}{\fk}\ot \id_\fk)
( \id _\ff \ot \tau_{\fk,\fk})(\lambda _\fk \ot \id_\fk)(\id -\tau)
+(\tau-\id)\eta_{\fk,\fk}.
\intertext{Since
$(\lma{}{\fk}\ot \id_\fk)
( \id _\ff \ot \tau_{\fk,\fk})(\lambda _\fk \ot \id_\fk)=\zeta_{\fk,\fk}$,
it follows 
from the definition of $\crmh$
and $\tau^2=\id$
that the last expression is equal to
}
&\quad (\id -\tau)(\beta_\fk \ot \id_\fk)
( \id \ot \tau)(\delta_\fk \ot \id_\fk)(\id -\tau) +(\tau _{\fk,\fk}-\id_{\fk\ot \fk})\crmh _{\fk,\fk}.
\end{align*}
Since $\fk $ is a Lie bialgebra in $(\lcrmC{\ff},\crmh)$,
\eqref{eq:bialgebra_compatibility} holds when composing from the left with $\pi_\fk \ot \pi_\fk$ and from the right with $\iota_\fk \ot \iota_\fk$.

Next we obtain from
\eqref{eq:beta} and \eqref{eq:delta} that
\begin{align*}
 (\pi_\ff \ot \pi_\fk)\delta \beta (\iota_\fk \ot \iota_\fk)
&=\lambda _\fk\,\pi_\fk \,\iota_\fk \,\beta_\fk 
=\lambda_\fk \,\beta_\fk=(\id \ot \beta_\fk)\lambda_{\fk \ot \fk},
\end{align*}
since $\beta _\fk:\fk\ot \fk \to \fk$
is a Lie comodule morphism by assumption.

On the other hand,
\begin{align}
\notag
(\pi_\ff \ot \pi_\fk)&(\id -\tau )
(\beta \ot \id)(\id \ot \tau) (\delta \ot \id )
(\id -\tau )(\iota_\fk \ot \iota_\fk)\\
\notag
&\qquad
+(\pi_\ff \ot \pi_\fk)(\tau -\id )\eta_{\fg,\fg}(\iota_\fk \ot \iota_\fk)\\
\notag
&=(\pi_\ff \ot \pi_\fk)(\id -\tau)(\beta \ot \id)
( \id \ot \tau)(\delta\iota_\fk \ot \iota_\fk)(\id -\tau)\\
\intertext{since $\eta $ is a natural transformation and $\pi_\ff \iota_\fk=0$. Since $\pi_\ff\beta=\beta_\ff (\pi_\ff\ot \pi_\ff)$ and since
$\tau (\beta \ot \id)(\id \ot \tau )=(\id \ot \beta )(\tau \ot \id )$, the latter expression can be rewritten as}
\notag
&=(\beta_\ff (\pi_\ff \ot \pi_\ff)\ot \pi_\fk)
(\id \ot \tau )(\delta\iota_\fk \ot \iota_\fk)(\id -\tau)\\
\notag
&\qquad
-(\pi_\ff \ot \pi_\fk)(\id \ot \beta )(\tau \delta \iota_\fk 
\ot \iota_\fk )(\id -\tau )
\\
\label{eq:bos1}
&=(\pi_\ff \ot \pi_\fk)(\id \ot \beta )(\delta\iota_\fk \ot \iota_\fk)(\id -\tau),
\end{align}
where the last equation follows from $\pi_\ff \iota_\fk=0$ and from $\tau \delta=-\delta$.
For the next reformulation we conclude first from \eqref{eq:delta} that $(\pi _\ff \ot \id )\delta \iota_\fk =(\pi_\ff \ot \iota_\fk \pi_\fk)\delta \iota_\fk $.
Hence the expression in \eqref{eq:bos1} is equal to
$$ (\id \ot \pi_\fk\beta )
(\id \ot \iota_\fk \ot \iota_\fk ) ((\pi_\ff \ot \pi_\fk )\delta\iota_\fk \ot \id)(\id -\tau)
=(\id _\ff\ot \beta_\fk )(\lambda_\fk \ot \id _\fk )(\id -\tau ).
$$
Since $\beta_\fk \tau=-\beta_\fk $, the latter expression is equal to $(\id \ot \beta_\fk )\lambda_{\fk \ot \fk }$.
Therefore  \eqref{eq:bialgebra_compatibility} holds when composing from the left with $\pi_\ff \ot \pi_\fk$ and from the right with $\iota_\fk \ot \iota_\fk$.

Very similarly we conclude that \eqref{eq:bialgebra_compatibility} holds
when composing from the left with $\pi_\fk \ot \pi_\fk$ and from the right with $\iota_\ff \ot \iota_\fk$.

Finally, Equations~\eqref{eq:beta} and \eqref{eq:delta} imply that
\begin{align*}
 (\pi_\ff \ot \pi_\fk)\delta \beta (\iota_\ff \ot \iota_\fk)
&= (\pi_\ff \ot \pi_\fk)\delta \iota_\fk \pi_\fk
\beta (\iota_\ff \ot \iota_\fk)=\lambda _\fk\,\lma{}{\fk }.
\end{align*}
On the other hand, the expression
\begin{align*}
 (\pi_\ff \ot \pi_\fk)&(\id -\tau )
(\beta \ot \id)(\id \ot \tau) (\delta \ot \id )
(\id -\tau )(\iota_\ff \ot \iota_\fk)\\
&\qquad +(\pi_\ff \ot \pi_\fk)
(\tau -\id )\eta_{\fg,\fg}(\iota_\ff \ot \iota_\fk )
\end{align*}
is the sum of the following five terms:
\begin{align*}
 &(\pi_\ff \ot \pi_\fk)
(\beta \ot \id)(\id \ot \tau) (\delta \ot \id )
(\iota_\ff \ot \iota_\fk)\\
&\quad =(\pi_{\ff }\beta \ot \pi_{\fk })(\iota_\ff \ot \iota_\fk \ot \iota _\ff)(\id _\ff\ot \tau_{\ff,\fk})(\delta_\ff \ot \id_\fk)=0,\\
 &(\pi_\ff \ot \pi_\fk)
\tau (\beta \ot \id)(\id \ot \tau) (\delta \ot \id )\tau (\iota_\ff \ot \iota_\fk)\\
 &\quad =(\pi_\ff \ot \pi_\fk)
(\id \ot \beta )(\tau \ot \id)(\id \ot \delta) (\iota_\ff \ot \iota_\fk)\\
&\quad = (\id_\ff \ot \alpha_\fk)(\tau_{\ff,\ff}\ot \id_\fk)(\id _\ff\ot \lambda_\fk),\\
 &-(\pi_\ff \ot \pi_\fk)
\tau (\beta \ot \id)(\id \ot \tau) 
(\delta \ot \id )(\iota_\ff \ot \iota_\fk)\\
&\quad =(\pi_\ff \ot \pi_\fk )(\id \ot \beta )
(\delta  \ot \id )(\iota_{\ff }\ot \iota_\fk)\\
&\quad =
(\id _\ff \ot \alpha_\fk ) (\delta _\ff \ot \id _\fk),\\
&-(\pi_\ff \ot \pi_\fk)
(\beta \ot \id)(\id \ot \tau) (\delta \ot \id )
\tau (\iota_\ff \ot \iota_\fk)=(\beta_\ff \ot \id_\fk )
(\id _\ff \ot \lambda_\fk),\\
&(\pi_\ff \ot \pi_\fk)
(\tau -\id )\eta_{\fg,\fg}(\iota_\ff \ot \iota_\fk )
=-\eta_{\fk,\fk}.
\end{align*}
Thus, by the crossed module axiom for $\fk$, \eqref{eq:bialgebra_compatibility} holds
when composing from the left with $\pi_\ff \ot \pi_\fk$ and from the right with $\iota_\ff \ot \iota_\fk$.
This finishes the proof of the theorem.
\end{proof}

\section{Examples}
\label{section:examples}

In this section we present some non-trivial curved Lie bialgebras. Each of these examples are related to certain color vector spaces.

Let $G$ be an abelian group. We omit the symbol for the group operation and write $gh$ for the product of two elements $g,h\in G$, $1$ for the neutral element of $G$, and $g^{-1}$ for the inverse of $g\in G$.
Let $\Bbbk $ be a field, $\Bbbk^\times $ its subgroup of units, and let $\chi :G\times G\to \Bbbk ^\times $ be an antisymmetric bicharacter, also called composition factor. The latter means that
\begin{align}
  \chi(g_1,g_2)\chi(g_2,g_1)=1, \quad
  \chi(g_1g_2,h)=\chi(g_1,h)\chi(g_2,h).
\end{align}
These equations imply in particular that $\chi(g,g)^2=1$ for all $g\in G$. Moreover, $\chi $ defines a group homomorphism $G\to \Bbbk^\times  $ via $g\mapsto \chi(g,g)$.

A $G$-graded vector space is
a vector space $V$ with a direct sum decomposition $V=\bigoplus_{g\in G}V_g$. A morphism between $G$-graded vector spaces $V,W$ is a linear map $f:V\to W$ with $f(V_g)\subset W_g$ for all $g\in G$.
The category of $G$-graded vector spaces over $\Bbbk $ together with an antisymmetric bicharacter $\chi :G\times G\to \Bbbk^\times $
is commonly known as the category of ($(G,\chi)$-)\textbf{color vector spaces}. It is a monoidal category, where
$$ (V\ot W)_g=\bigoplus_{h\in G}V_h\ot W_{h^{-1}g} $$
for all $G$-graded vector spaces $V,W$. Moreover, the presence of the antisymmetric bicharacter allows to attach a non-trivial symmetry to the category:
$$ \tau_{V,W}:V\ot W\to W\ot V,\quad \tau_{V,W}(v\ot w)=\chi(g,h)w\ot v
$$
for all $v\in V_g$, $w\in W_h$.
We view the category of $(G,\chi)$-color vector spaces as a Cartier category with symmetry $\tau $ and with zero infinitesimal braiding, and abbreviate it by $\cC$.

Each object $\ff\in \cC$ is a Lie algebra with zero bracket, called abelian Lie algebra. Similarly, each object $\ff\in \cC$ is a Lie coalgebra with zero cobracket. We then say that $\ff$ is coabelian. Our examples in this section will be (curved) Lie bialgebras in the category of crossed modules (see Proposition~\ref{pr:cmfriendly}) over an abelian coabelian Lie bialgebra.
For the discussion of the Laistrygonian examples the following lemma will be useful.

\begin{lem}
\label{lem:coabcoaction}
Let $\ff$ be a trivially $G$-graded coabelian Lie coalgebra in $\cC$. Let $V\in \cC$ and let $(x_i)_{i\in I}$ be a vector space basis of $V$ consisting of homogeneous elements with respect to the $G$-grading. Then for each family $(f_i)_{i\in I}$ of elements in $\ff$, the map
$$ \lambda:V\to \ff \ot V,\quad
\lambda(x_i)=f_i\ot x_i
$$
defines a left Lie comodule structure over $\ff$ on $V$.
\end{lem}

\begin{proof}
  The assumptions imply that $\lambda $ is a morphism in $\cC$, $\delta \ot \id:\ff \ot V\to \ff \ot \ff \ot V$ is zero, and
  $$ ((\id -\tau)\ot \id)(\id \ot \lambda)\lambda=0.
  $$
  This implies the claim.
\end{proof}

The Nichols algebras in \cite{MR4298502} indicate the existence of several finite-dimensional  (curved) Lie bialgebras
in the category of crossed modules over an abelian coabelian Lie bialgebra. We discuss here only two single examples (the Jordan and the super Jordan plane) and a series (the Laistrygonian examples) in detail. Other candidates are the super Laistrygonians as well as the (super) Endymion and Poseidon examples. Our presentation here is aimed to be explicit rather than being efficient. Nonetheless, we strongly believe that by adding more theory an appealing efficient presentation of large classes of examples is possible.

\subsection{The Jordan plane}

We start with a relatively simple example, where the benefit of the combination of so much fairly trivial structure may not be fully convincing at the first moment. Nevertheless, this example will appear later as a subobject in more complicated examples. Therefore it is worthwhile to study it separately.  

Assume that $G=\langle g\rangle \cong (\ndZ,+)$ and $\chi(g,g)=1$. Let $\cC$ be the category of $(G,\chi)$-color vector spaces (with zero infinitesimal braiding).
Let $\ff=\Bbbk s$ be a one-dimensional $(G,\chi)$-color vector space of $G$-degree $1$, with the abelian and coabelian Lie bialgebra structure. 
Let $J\in \cC $ be a two-dimensional homogeneous $G$-graded vector space of $G$-degree $g$ and let $x_1,x_2$ be a basis of $J$. 

\begin{pro} \label{pr:Jordan}
  Let $J\in \lcrmC{\ff}$ with the following action and coaction of $\ff$:
$$ \lambda(x)=s\ot x \quad
\text{for all $x\in J$,}
$$
and
$$
s\cdot x_1=0,\quad s\cdot x_2=x_1,
$$
where $s\cdot x=\alpha(s\ot x)$ for all $x\in J$.
Then $J$ is a curved Lie bialgebra in $(\lcrmC{\ff},\crmh)$ with zero bracket, zero cobracket, and the natural infinitesimal braiding $\crmh $ of $\lcrmC{\ff}$.
\end{pro}

It is illuminating to put the definition of the curved Lie bialgebra $J$ next to the description of the Nichols algebra of the Jordan plane in \cite[Prop.\,3.2.1]{MR4298502}.  We call $J$ the Lie bialgebra of the Jordan plane.

\begin{proof}
Since $\tau(s\ot s)=s\ot s$, it is easy to see that $(J,\alpha)\in \lmC{\ff}$ and $(J,\lambda )\in \lcmC{\ff}$.
Since $\ff$ is abelian and coabelian, the crossed module axiom~\eqref{eq:lcrm} is equivalent to
$$ \lambda (s\cdot x)=
(\id \ot \alpha)(\tau_{\ff,\ff}\ot \id)(\id \ot \lambda )(s\ot x)\quad 
\text{for all $x\in J$,}
$$
which is clearly satisfied. Thus $J\in \lcrmC{\ff}$.

The crossed module $J\in \lcrmC{\ff}$ is a Lie algebra with zero bracket and a Lie coalgebra with zero cobracket.
The infinitesimal braiding of $J$ in $\lcrmC{\ff}$
in the notation of Lemma~\ref{le:zeta} is
$$ \crmh_{J,J}(x\ot y)=(\zeta \tau +\tau \zeta )(x\ot y)=s\cdot x\ot y
+x\ot s\cdot y $$
for all $x,y\in J$. Since $x_1\ot x_1,x_1\ot x_2+x_2\ot x_1,x_2\ot x_2\in \ker (\tau-\id )$, it follows that
$$ (\tau -\id)(J\ot J)=\Bbbk (x_1\ot x_2-x_2\ot x_1),\qquad
(\tau -\id)\crmh _{J,J}=0. $$
Therefore the zero bracket and zero cobracket satisfy the Lie bialgebra axiom \eqref{eq:bialgebra_compatibility}, and hence $J$ is a curved Lie bialgebra in $(\lcrmC{\ff},\crmh)$. 
\end{proof}

\subsection{The super Jordan plane}

This is our first example of a nonabelian non-coabelian curved Lie bialgebra structure. It is illuminating to put it next to the description of the Nichols algebra of the super Jordan plane in \cite[Prop.\,3.3.1]{MR4298502}.
We do not put any assumptions on the field $\Bbbk$, not even on its characteristic, but make a comment on this after explaining the example in Proposition~\ref{pr:superJordan} below.

Assume that $G=\langle g\rangle \cong (\ndZ,+)$ and $\chi(g,g)=-1$.
Let $\ff=\Bbbk s$ be a one-dimensional $(G,\chi)$-color vector space of $G$-degree $1$, with the abelian and coabelian Lie bialgebra structure.
Let $J^-$ be a four-dimensional homogeneous $G$-graded vector space with basis
$$ x_{11}, x_{12},x_{21},x_{22}, $$
where $x_{ij}\in J^-_{g^i}$ for all $i,j\in \{1,2\}$. For convenience, we write $x_{i0}=0$ for all $i\in \{1,2\}$.
When comparing $J^-$ with the super Jordan plane in 
\cite[Prop.\,3.3.1]{MR4298502},
the generators $x_{1i}$ with $i\in \{1,2\}$ here should be identified with the generators $x_i$ there.

The definition of $J^-$ implies that
$$ \tau(x_{ij}\ot x_{kl})
=\chi(g^i,g^k)x_{kl}\ot x_{ij}=(-1)^{ik}x_{kl}\ot x_{ij} $$
for all $i,j,k,l\in \{1,2\}$.

\begin{pro} \label{pr:superJordan}
  Let $J^-\in \lcrmC{\ff}$ with the following action and coaction of $\ff$:
\begin{align}
  \lambda(x_{ij})&=is\ot x_{ij} \quad \text{for all $i,j$,}\\
  s\cdot x_{ij}&=x_{i,j-1}
  \quad \text{for all $i,j$},
\end{align}
where $s\cdot x=\alpha(s\ot x)$ for all $x\in J^-$.
Then $J^-$ is a curved Lie bialgebra in $(\lcrmC{\ff},\crmh)$ with
the natural (non-zero) infinitesimal braiding $\crmh $ of $\lcrmC{\ff}$ and
the following 
bracket $\beta $ and cobracket $\delta $:
\begin{align}
  \beta (x_{1i}\ot x_{1j}) &=(i+j-2)x_{2,i+j-2},\\
  \beta (x_{2i}\ot x) &=
  \beta(x\ot x_{2i})=0,\\
  \delta (x_{1i})&=0,\\
  \delta(x_{2i})
  &=(\tau -\id)(x_{11}\ot x_{1i})
  =-x_{11}\ot x_{1i}-x_{1i}\ot x_{11}
\end{align}
for all $i,j\in \{1,2\}$, $x\in J^-_{g^2}$.
\end{pro}

Using Kronecker's delta,
$$ \Kd_{ij}=\begin{cases} 1 & \text{if $i=j$,}\\ 0 & \text{if $i\ne j$,}
\end{cases}
$$
the definition of the bracket and cobracket can also be written as
\begin{align}
  \beta(x_{ij}\ot x_{kl})&=
  \Kd_{i1}\Kd_{k1}(j+l-2)x_{2,j+l-2},\\
  \delta(x_{ij})&=\Kd_{i2}(\tau-\id)
  (x_{11}\ot x_{1j})
  =-\Kd _{i2}x_{11}\ot x_{1j}
  -\Kd _{i2}x_{1j}\ot x_{11}
\end{align}
for all $i,j,k,l\in \{1,2\}$.

We call $J^-$ the Lie bialgebra of the super Jordan plane.

Recall that we did not put any assumption on the field $\Bbbk$. Assume now that the characteristic of $\Bbbk$ is $2$. Then $\Bbbk x_{21}$ is a one-dimensional subobject of $J^-$ in $\lcrmC{\ff}$ and a Lie ideal of $J^-$. Moreover, $\delta (x_{21})=0$, which allows to take the quotient $J^-/\Bbbk x_{21}$. This does not work for other fields.

\begin{proof}
It is fairly clear that $(J^-,\alpha)\in \lmC{\ff}$ and $(J^-,\lambda)\in \lcmC{\ff}$.
Since $\ff $ is abelian and coabelian,
the crossed module axion \eqref{eq:lcrm} is equivalent to
$$ \lambda (s\cdot x)=
(\id \ot \alpha )(\tau \ot \id)(\id \ot \lambda )(s\ot x )\quad
\text{for all $x\in J^-$}, $$
which follows directly from $\alpha (s\ot J^-_{g^i})\subseteq J^-_{g^i}$ for all $i\in \{1,2\}$.
Thus $J^-\in \lcrmC{\ff}$.

It is clear that $\beta $ and $\delta $ are morphisms in $\lcmC{\ff}$. Moreover,
for all $i,j,k,l\in \{1,2\}$ we obtain that
\begin{align*}
s\cdot \beta(x_{ij}\ot x_{kl})&=
s\cdot \Kd_{i1}\Kd_{k1}(j+l-2)x_{2,j+l-2}
=\Kd_{i1}\Kd_{k1}\Kd_{j2}\Kd_{l2}2x_{21},
\\
\beta(s\cdot (x_{ij}\ot x_{kl}))&=
\beta(\Kd_{j2}x_{i1}\ot x_{kl}
+\Kd_{l2}x_{ij}\ot x_{k1})
=\Kd_{i1}\Kd_{k1}\Kd_{j2}\Kd_{l2}
2x_{21},\\
s\cdot \delta(x_{ij})&=s\cdot
(-\Kd_{i2}x_{11}\ot x_{1j}
-\Kd_{i2}x_{1j}\ot x_{11})
=-2\Kd_{i2}\Kd_{j2}x_{11}\ot x_{11},\\
\delta(s\cdot x_{ij})
&=\delta(\Kd_{j2}x_{i1})
=\Kd_{j2}\Kd_{i2}(-2)x_{11}\ot x_{11}.
\end{align*}
Hence $\beta$ and $\delta$ are morphisms in $\lcrmC{\ff}$.

The antisymmetry of $\beta $ and the co-antisymmetry of $\delta $ are fairly obvious. Since the image of $\beta \ot \id -(\id \ot \beta)(\id-\tau_{12})$ is contained in $J^-_{g^2}\ot J^-+J^-\ot J^-_{g^2}\subseteq \ker \beta$, $\beta $ satisfies  the Jacobi identity. Similarly,
$$ \delta (J^-)\subseteq J^-_g\ot J^-_g
\subseteq \ker \big(\delta \ot \id -(\id -\tau_{12})(\id \ot \delta)\big). $$
Thus $(J^-,\beta)$ is a Lie algebra in $\lcrmC{\ff}$ and $(J^-,\delta)$ is a Lie coalgebra in $\lcrmC{\ff}$.

The infinitesimal braiding of $J^-$ in $\lcrmC{\ff}$
in the notation of Lemma~\ref{le:zeta} is
\begin{align*}
  \crmh_{J^-,J^-}(x_{ij}\ot x_{kl})&=(\zeta \tau +\tau \zeta )(x_{ij}\ot x_{kl})\\
  &=ks\cdot x_{ij}\ot x_{kl}+x_{ij}\ot is\cdot x_{kl}\\
  &=k\Kd_{j2}x_{i1}\ot x_{kl}
  +i\Kd_{l2}x_{ij}\ot x_{k1}
\end{align*}
for all $i,j,k,l\in \{1,2\}$. In particular, $(\tau -\id)\crmh _{J^-,J^-}\ne 0$.

We check the Lie bialgebra axiom
\eqref{eq:bialgebra_compatibility} for $\beta $ and $\delta$.

For all $i,j,k,l\in \{1,2\}$
we obtain directly from the definitions that
\begin{align*}
    \delta \beta(x_{ij}\ot x_{kl})
    &=\Kd_{i1}\Kd_{k1}(j+l-2)
    \delta(x_{2,j+l-2})\\
    &=\Kd_{i1}\Kd_{k1}(j+l-2)
    (\tau-\id)(x_{11}\ot x_{1,j+l-2}).
\end{align*}
Moreover,
\begin{align*}
  (\id-\tau)&(\beta \ot \id)(\id \ot \tau)(\delta \ot \id)(\id-\tau)(x_{ij}\ot x_{kl})\\
  &=(\id-\tau)(\beta \ot \id)(\id \ot \tau)(\delta \ot \id)(x_{ij}\ot x_{kl}-(-1)^{ik}x_{kl}\ot x_{ij})\\
  &=(\id-\tau)(\beta \ot \id)(\id \ot \tau)(-\Kd_{i2}x_{11}\ot x_{1j}\ot x_{kl}-\Kd_{i2}x_{1j}\ot x_{11}\ot x_{kl}\\
  &\qquad \qquad \qquad
  +\Kd_{k2}x_{11}\ot x_{1l}\ot x_{ij}
  +\Kd_{k2}x_{1l}\ot x_{11}\ot x_{ij})\\
  &=(\id-\tau)(\Kd_{i2}\Kd_{k1}\Kd_{l2}x_{21}\ot x_{1j}
  +\Kd_{i2}\Kd_{k1}(j+l-2)x_{2,j+l-2}\ot x_{11}\\
  &\qquad \qquad \qquad
  -\Kd_{i1}\Kd_{j2}\Kd_{k2}x_{21}\ot x_{1l}
  -\Kd_{i1}\Kd_{k2}(j+l-2)x_{2,j+l-2}\ot x_{11}
  ).
\end{align*}
Since the three terms obtained are of the form 
$\tau-\id$ evaluated on something, 
the Lie bialgebra axiom for $\beta$ and $\delta$ is equivalent to
\begin{equation}
\label{eq:Liebialg-superJordan}
\begin{aligned}
    &(j+l-2)\Kd_{i1}\Kd_{k1}x_{11}\ot x_{1,j+l-2}+\Kd_{i2}\Kd_{k1}\Kd_{l2}x_{21}\ot x_{1j}\\
    &\quad +\Kd_{i2}\Kd_{k1}
    (j+l-2)x_{2,j+l-2}\ot x_{11}
    -\Kd_{i1}\Kd_{j2}\Kd_{k2}x_{21}\ot x_{1l}\\
    &\quad
    -\Kd_{i1}\Kd_{k2}(j+l-2)x_{2,j+l-2}\ot x_{11}
    -k\Kd_{j2}x_{i1}\ot x_{kl}
    -i\Kd_{l2}x_{ij}\ot x_{k1}\\
    &\quad
    \in \ker(\tau-\id).
\end{aligned}
\end{equation}
For $i=k=2$, Equation~\eqref{eq:Liebialg-superJordan} is equivalent to
$$ -2\Kd_{j2}x_{21}\ot x_{2l}
-2\Kd_{l2}x_{2j}\ot x_{21}\in \ker(\tau-\id),$$
which is clear whenever $j=1$ or $l=1$ or $j=l$, and hence in all cases.

For $i=2$, $k=1$, 
Equation~\eqref{eq:Liebialg-superJordan} is equivalent to
\begin{align*}
    &\Kd_{l2}x_{21}\ot x_{1j}
    +(j+l-2)x_{2,j+l-2}\ot x_{11}
    -\Kd_{j2}x_{21}\ot x_{1l}
    -2\Kd_{l2}x_{2j}\ot x_{11}\\
    &\quad \in \ker(\tau-\id).
\end{align*}
One checks for all possible pairs $(j,l)$ that the expression on the left hand side of the relation is in fact already $0$.

For $i=1$, $k=2$,
Equation~\eqref{eq:Liebialg-superJordan} is equivalent to
\begin{align*}
    &-\Kd_{j2}x_{21}\ot x_{1l}
    -(j+l-2)x_{2,j+l-2}\ot x_{11}
    -2\Kd_{j2}x_{11}\ot x_{2l}
    -\Kd_{l2}x_{1j}\ot x_{21}\\
    &\quad
    \in \ker(\tau-\id).
\end{align*}
By going through all four possibilities for the pair $(j,l)$ 
one checks that the relation holds.

Finally, for $i=1$, $k=1$, 
Equation~\eqref{eq:Liebialg-superJordan} is equivalent to
\begin{align*}
    &(j+l-2)x_{11}\ot x_{1,j+l-2}
    -\Kd_{j2}x_{11}\ot x_{1l}
    -\Kd_{l2}x_{1j}\ot x_{11}
    \in \ker(\tau-\id).
\end{align*}
Here, if $(j,l)\ne (2,2)$, then the expression on the left hand side of the relation is already $0$. If however $j=l=2$, then the condition simplifies to
$$ x_{11}\ot x_{12}-x_{12}\ot x_{11}\in \ker (\tau-\id),$$
which is easily checked.

We conclude that $(J^-,\beta,\delta)$ satisfies the Lie bialgebra axiom and hence it is a curved Lie bialgebra in $(\lcrmC{\ff},\crmh)$, and the proof is completed.
\end{proof}

\subsection{The Laistrygonian Lie bialgebras}

In this section we use a larger abelian coabelian Lie bialgebra $\ff$, but the general approach remains similar. We identify an infinite family of curved Lie bialgebras, see Proposition~\ref{pr:Lone} below, each of them containing the Jordan plane as a subobject.

Assume that $G=\langle g,h\rangle \cong (\ndZ^2,+)$ and
$$ \chi(g,g)=\chi(h,h)=1. $$
Let $\cC$ be the category of $(G,\chi)$-color vector spaces
(with zero infinitesimal braiding)
over a field $\Bbbk$ of characteristic $\ne2$.
Let
$$ \ff=\ff_1=\Bbbk s+\Bbbk t $$
be a two-dimensional $(G,\chi)$-color vector space of $G$-degree $1$, with the abelian and coabelian Lie bialgebra structure:
$$ \beta_\ff=0,\quad \delta_\ff=0. $$
Let $\cG\in \mathbb{N}_0$ and
let
\begin{gather*}
  \Lone =\Lone_g\oplus \bigoplus_{k=0}^\cG \Lone_{g^kh}
\in \cC,\\
\Lone_g=\Bbbk x_1+\Bbbk x_2,\quad
\Lone_{g^kh}=\Bbbk z_k \quad \text{for all $0\le k\le \cG$,}
\end{gather*}
be a $\cG+3$-dimensional $G$-graded vector space.
Let $\Loone $ be the subspace of $\Lone $ spanned by $x_1$ and $z_k$, $0\le k\le \cG$. Clearly, $\Loone \in \cC$.

\begin{pro} \label{pr:Lone}
  Let $\cG\in \mathbb{N}_0$. Then $\Lone \in \lcrmC{\ff}$ with the following action and coaction of $\ff$:
$$ \lambda(x_1)=s\ot x_1,\quad
\lambda (x_2)=s\ot x_2,\quad
\lambda (z_k)=(ks+t)\ot z_k
$$
for all $0\le k\le \cG$,
and
$$
s\cdot x_2=x_1,\quad t\cdot x_2=-\frac{\cG}2x_1,
\quad s\cdot x=t\cdot x=0
$$
for all $x\in \Loone $,
where $f\cdot x=\alpha(f\ot x)$ for all $f\in \ff $ and $x\in \Lone $.
Moreover, $\Lone $ is a curved Lie bialgebra in the category $(\lcrmC{\ff},\crmh )$ with the natural (non-zero) infinitesimal braiding $\crmh $ of $\lcrmC{\ff}$, bracket $\beta $ and cobracket $\delta $, where
\begin{align*}
    \beta (x\ot y)&=0 \quad \text{for all $x,y\in \Loone $,}\\
    \beta(x_2\ot x_1) &=\beta(x_2\ot x_2) =\beta(x_1\ot x_2)=0,\\
    \beta(x_2\ot z_k) &=-\beta \tau (x_2\ot z_k) =z_{k+1},\\
    \delta(x_1)&=\delta(x_2)=0,\\
  \delta (z_k)&=
  \frac {k(k-1-\cG)}2 (\tau-\id)(x_1\ot z_{k-1})
\end{align*}
for all $0\le k\le \cG$ (with the convention $z_{\cG+1}=0$).
\end{pro}

It is illuminating to put the Lie bialgebras $\Lone $ next to the description of the Laistrygonian Nichols algebras in 
\cite[Prop.\,4.3.5]{MR4298502}.
We call the curved Lie bialgebras $\Lone $ Laistrygonian Lie bialgebras.

\begin{proof}
By Lemma~\ref{lem:coabcoaction},
$\lambda $ defines a Lie coaction
of $\ff $ on $\Lone $. It is easy
to check that the action of $s$
and of $t$ on $\Lone $ commute,
and hence $\alpha $ defines
a Lie module structure of $\ff$
on $\Lone $.
The crossed module axiom \eqref{eq:lcrm}
applied to $f\ot x$ with $f\in \ff$ and $x\in \Loone $
is satisfied trivially, since all terms are zero. The axiom for $f\ot x_2$ with $f\in \{s,t\}$ is easily checked.
Thus $\Lone \in \lcrmC{\ff}$ and $\Loone \in \lcrmC{\ff}$.

  Clearly, $\beta $ is a morphism in $\cC$ and $\beta \tau =-\beta $. One checks quickly that $\beta $ is a morphism in $\lcrmC{\ff}$. A straightforward calculation shows that $\beta $ satisfies the Jacobi identity. For example, for all $x\in \Loone $ one obtains that
\begin{align*}
    \beta(\id \ot \beta)&(\id +\tau_{23}\tau_{12}+\tau_{12} \tau_{23})(x_2\ot x_2\ot x)\\
    &=\beta (x_2\ot \beta (\id +\tau )(x_2\ot x))
    +\beta (\id \ot \beta )\tau_{12}\tau_{23}(x_2\ot x_2\ot x)=0
\end{align*}
by the antisymmetry of $\beta $
and since $\beta (x_2\ot x_2)=0$. Hence $(\Lone ,\beta )$ is a Lie algebra in $\lcrmC{\ff}$.

It is fairly obvious that $\delta $ is a morphism in $\lcrmC{\ff}$.
By definition, $\tau \delta=-\tau $.
The co-Jacobi identity clearly holds when evaluated at $x_1$ or $x_2$.
On $z_k$ with $0\le k\le \cG$ it follows by using the definition of $\delta$ and the equation $(\tau -\id)(x_1\ot x_1)=0$. Hence $(\Lone,\delta)$ is a Lie coalgebra in $\lcrmC{\ff}$.

For the infinitesimal braiding we obtain from
Lemma~\ref{le:zeta} that
\begin{align}
    \crmh(x\ot y)&=0\quad
    \text{for all $x,y\in \Loone $,}\\
    \crmh (x_2\ot x_2)&=x_1\ot x_2+x_2\ot x_1,\\
    \crmh (x_2\ot z_k)&=(k-\frac\cG 2)x_1\ot z_k
\end{align}
for all $0\le k\le \cG$.
In particular, $\crmh $ is non-zero.

Finally, we verify the Lie bialgebra axiom \eqref{eq:bialgebra_compatibility} on the given basis vectors of $\Lone$.
To do so, we introduce first the total order
$$ x_1 < x_2 < z_0 < z_1 <
\cdots < z_\cG $$
on our basis. Since both sides of
\eqref{eq:bialgebra_compatibility}
are invariant under multiplication with $-\tau $ from the right, it suffices to check the Lie bialgebra axiom for tensors $x\ot y$ of basis vectors $x,y$ with $x\le y$. 

First, the Lie bialgebra axiom has been already verified on $\Bbbk x_1+\Bbbk x_2$ when studying the Lie bialgebra of the Jordan plane.
Moreover,
\[ \delta(\Loone)\subseteq \Loone \ot \Loone ,\]
and
$\beta (x\ot y)=0$ and $\crmh (x\ot y)=0$ for all $x,y\in \Loone $. Thus
$\delta \beta(x\ot y)=\delta(0)=0$,
and
\begin{align*}
    &(\id -\tau )(\beta \ot \id)(\id \ot \tau)(\delta \ot \id )(\id -\tau )(x\ot y)\\
    &\quad
    \in (\id -\tau )(\beta \ot \id)(\Loone \ot \Loone \ot \Loone )=0
\end{align*}
for all $x,y\in \Loone $.

Lastly, for all $0\le k\le \cG$
(with $z_{-1}=z_{\cG+1}=0$) we obtain that
\begin{align*}
  \delta \beta (x_2\ot z_k)
  &=\delta (z_{k+1})
  =\frac{(k+1)(k-\cG)}2(\tau -\id)
  (x_1\ot z_k),\\
  (\tau -\id)\crmh (x_2\ot z_k)
  &=\big(k-\frac \cG 2\big)(\tau -\id )(x_1\ot z_k),
\end{align*}
and
\begin{align*}
    (\id -\tau )& (\beta \ot \id)
    (\id \ot \tau)(\delta \ot \id )
    (\id -\tau )(x_2\ot z_k)\\
    &=
    (\id -\tau )(\beta \tau \ot \id)
    (\id \ot \delta )
    (-x_2\ot z_k)\\
    &=
    \frac{k(k-1-\cG)}2(\id -\tau ) (\beta \ot \id)
    (\tau_{23}-\id )
    (x_2\ot x_1\ot z_{k-1})\\
&=
    \frac{k(k-1-\cG)}2(\id -\tau ) \tau (\id \ot \beta )
    (\tau \ot \id )
    (x_2\ot x_1\ot z_{k-1})\\
&=
    \frac{k(k-1-\cG)}2(\tau -\id ) (x_1\ot z_k).
\end{align*}
These equations imply directly that \eqref{eq:bialgebra_compatibility}
is fulfilled, completing the proof of Proposition~\ref{pr:Lone}.
\end{proof}

\subsection*{Acknowledgements}

This work was partially supported by 
the project OZR3762 of Vrije Universiteit Brussel. We thank
Pavel Etingof for pointing out the work \cite{MR1331627} 
of Cartier on infinitesimal braidings. We also thank
the referee for comments and suggestions.

\bibliographystyle{abbrv}
\bibliography{refs}

\def\cprime{$'$}
\begin{thebibliography}{10}

\bibitem{MR1227871}
N.~Andruskiewitsch.
\newblock Lie superbialgebras and {P}oisson-{L}ie supergroups.
\newblock {\em Abh. Math. Sem. Univ. Hamburg}, 63:147--163, 1993.

\bibitem{MR4298502}
N.~Andruskiewitsch, I.~Angiono, and I.~Heckenberger.
\newblock On finite {GK}-dimensional {N}ichols algebras over abelian groups.
\newblock {\em Mem. Amer. Math. Soc.}, 271(1329):ix+125, 2021.

\bibitem{MR3322335}
I.~Buchberger and J.~Fuchs.
\newblock On the {K}illing form of {L}ie algebras in symmetric ribbon
  categories.
\newblock {\em SIGMA Symmetry Integrability Geom. Methods Appl.}, 11:Paper 017,
  21, 2015.

\bibitem{MR1331627}
P.~Cartier.
\newblock Construction combinatoire des invariants de {V}assiliev-{K}ontsevich
  des n\oe uds.
\newblock In {\em R.{C}.{P}. 25, {V}ol. 45 ({F}rench) ({S}trasbourg,
  1992--1993)}, volume 1993/42 of {\em Pr\'{e}publ. Inst. Rech. Math. Av.},
  pages 1--10. Univ. Louis Pasteur, Strasbourg, 1993.

\bibitem{MR688240}
V.~G. Drinfel\cprime~d.
\newblock Hamiltonian structures on {L}ie groups, {L}ie bialgebras and the
  geometric meaning of classical {Y}ang-{B}axter equations.
\newblock {\em Dokl. Akad. Nauk SSSR}, 268(2):285--287, 1983.

\bibitem{MR1047964}
V.~G. Drinfel\cprime~d.
\newblock Quasi-{H}opf algebras.
\newblock {\em Algebra i Analiz}, 1(6):114--148, 1989.

\bibitem{MR2173841}
B.~Enriquez.
\newblock A cohomological construction of quantization functors of {L}ie
  bialgebras.
\newblock {\em Adv. Math.}, 197(2):430--479, 2005.

\bibitem{MR2407849}
B.~Enriquez and G.~Halbout.
\newblock Quantization of {$\Gamma$}-{L}ie bialgebras.
\newblock {\em J. Algebra}, 319(9):3752--3769, 2008.

\bibitem{MR2630065}
B.~Enriquez and G.~Halbout.
\newblock Quantization of coboundary {L}ie bialgebras.
\newblock {\em Ann. of Math. (2)}, 171(2):1267--1345, 2010.

\bibitem{MR2629982}
B.~Enriquez and G.~Halbout.
\newblock Quantization of quasi-{L}ie bialgebras.
\newblock {\em J. Amer. Math. Soc.}, 23(3):611--653, 2010.

\bibitem{MR3761992}
P.~Etingof.
\newblock Koszul duality and the {PBW} theorem in symmetric tensor categories
  in positive characteristic.
\newblock {\em Adv. Math.}, 327:128--160, 2018.

\bibitem{MR1403351}
P.~Etingof and D.~Kazhdan.
\newblock Quantization of {L}ie bialgebras. {I}.
\newblock {\em Selecta Math. (N.S.)}, 2(1):1--41, 1996.

\bibitem{MR1669953}
P.~Etingof and D.~Kazhdan.
\newblock Quantization of {L}ie bialgebras. {II}, {III}.
\newblock {\em Selecta Math. (N.S.)}, 4(2):213--231, 233--269, 1998.

\bibitem{MR1771217}
P.~Etingof and D.~Kazhdan.
\newblock Quantization of {L}ie bialgebras. {IV}. {T}he coinvariant
  construction and the quantum {KZ} equations.
\newblock {\em Selecta Math. (N.S.)}, 6(1):79--104, 2000.

\bibitem{MR1771218}
P.~Etingof and D.~Kazhdan.
\newblock Quantization of {L}ie bialgebras. {V}. {Q}uantum vertex operator
  algebras.
\newblock {\em Selecta Math. (N.S.)}, 6(1):105--130, 2000.

\bibitem{MR2452604}
P.~Etingof and D.~Kazhdan.
\newblock Quantization of {L}ie bialgebras. {VI}. {Q}uantization of generalized
  {K}ac-{M}oody algebras.
\newblock {\em Transform. Groups}, 13(3-4):527--539, 2008.

\bibitem{MR2264064}
N.~Geer.
\newblock Etingof-{K}azhdan quantization of {L}ie superbialgebras.
\newblock {\em Adv. Math.}, 207(1):1--38, 2006.

\bibitem{MR3070680}
I.~Goyvaerts and J.~Vercruysse.
\newblock A note on the categorification of {L}ie algebras.
\newblock In {\em Lie theory and its applications in physics}, volume~36 of
  {\em Springer Proc. Math. Stat.}, pages 541--550. Springer, Tokyo, 2013.

\bibitem{MR4164719}
I.~Heckenberger and H.-J. Schneider.
\newblock {\em Hopf algebras and root systems}, volume 247 of {\em Mathematical
  Surveys and Monographs}.
\newblock American Mathematical Society, Providence, RI, [2020] \copyright
  2020.

\bibitem{MR4008967}
B.~Hurle and A.~Makhlouf.
\newblock Quantization of color {L}ie bialgebras.
\newblock In {\em Geometric and harmonic analysis on homogeneous spaces},
  volume 290 of {\em Springer Proc. Math. Stat.}, pages 71--94. Springer, Cham,
  2019.

\bibitem{MR1321145}
C.~Kassel.
\newblock {\em Quantum groups}, volume 155 of {\em Graduate Texts in
  Mathematics}.
\newblock Springer-Verlag, New York, 1995.

\bibitem{MR1744574}
S.~Majid.
\newblock Braided-{L}ie bialgebras.
\newblock {\em Pacific J. Math.}, 192(2):329--356, 2000.

\bibitem{MR3839605}
Y.~I. Manin.
\newblock {\em Quantum groups and noncommutative geometry}.
\newblock CRM Short Courses. Centre de Recherches Math\'{e}matiques,
  [Montreal], QC; Springer, Cham, second edition, 2018.
\newblock With a contribution by Theo Raedschelders and Michel Van den Bergh.

\bibitem{MR594993}
W.~Michaelis.
\newblock Lie coalgebras.
\newblock {\em Adv. in Math.}, 38(1):1--54, 1980.

\bibitem{MR2734334}
F.~Montaner, A.~Stolin, and E.~Zelmanov.
\newblock Classification of {L}ie bialgebras over current algebras.
\newblock {\em Selecta Math. (N.S.)}, 16(4):935--962, 2010.

\bibitem{MR3154807}
D.~A. Rumynin.
\newblock Lie algebras in symmetric monodial categories.
\newblock {\em Sibirsk. Mat. Zh.}, 54(5):1128--1149, 2013.

\bibitem{MR725413}
M.~A. Semenov-Tyan-Shanski\u{\i}.
\newblock What a classical {$r$}-matrix is.
\newblock {\em Funktsional. Anal. i Prilozhen.}, 17(4):17--33, 1983.

\bibitem{MR1800719}
M.~Takeuchi.
\newblock Survey of braided {H}opf algebras.
\newblock In {\em New trends in {H}opf algebra theory ({L}a {F}alda, 1999)},
  volume 267 of {\em Contemp. Math.}, pages 301--323. Amer. Math. Soc.,
  Providence, RI, 2000.

\end{thebibliography}

\end{document}